\documentclass{article}
\usepackage[utf8]{inputenc}
\usepackage[T1]{fontenc}
\usepackage[english]{babel}
\usepackage{amsmath}
\usepackage{amsthm}
\usepackage{amsfonts}
\usepackage{amssymb}
\usepackage{authblk}
\usepackage{cases}
\usepackage{hyperref}
\usepackage{esvect}
\usepackage{bbm}
\usepackage{geometry}[margin=3in]
\usepackage{accents}
\usepackage{graphicx}
\usepackage{caption}
\usepackage{bbm}
\usepackage{MnSymbol}
\usepackage{relsize}
 \usepackage{url}

\newtheorem{thm}{Theorem}[section]
\newtheorem*{thm*}{Theorem}

\newtheorem*{theoremain}{Main Theorem}

\newtheorem{lemma}[thm]{Lemma}
\newtheorem{cor}[thm]{Corollary}
\newtheorem{rmk}[thm]{Remark}
\newtheorem{prop}[thm]{Proposition}

\newtheorem{definition}[thm]{Definition}
\newtheorem{conj}[thm]{Conjecture}
\numberwithin{equation}{section}
\newtheorem*{question}{Question}

\usepackage[english]{babel}

\usepackage{csquotes}

\usepackage{xcolor}
 \usepackage{hyperref}
          \hypersetup{colorlinks, citecolor=blue, filecolor=blue, linkcolor=black, urlcolor=blue}
     \usepackage{soul}
     \usepackage[numbers]{natbib}

\newcommand{\tw}{\mathcal{B}}
\newcommand{\twc}{\mathcal{T}}
\newcommand{\n}{\vec{\textbf{n}}}
\newcommand{\tpn}{t_p^{\n}}
\newcommand{\spn}{\mathcal{S}_{\n,p}}
\newcommand{\Bir}{\mathcal{O}}
\newcommand{\f}{\mathcal{F}}
\newcommand{\col}{\mathrm{col}}
\newcommand{\one}{{\textbf{1}}}

\newcommand{\toral}{\mathcal{R}_{\mathbb{T}^d}}

\title{Nondegeneracy of the spectrum of the twisted cocycle for interval exchange transformations}
\author{Hesam Rajabzadeh and Pedram Safaee}

 \AtEndDocument{\bigskip\bigskip\bigskip{
 \textsc{H. Rajabzadeh,}\par
  \textsc{School of Mathematics, Institute for Research in Fundamental Sciences (IPM),}\par{ P.O. Box 19395-5746, Tehran, Iran.} \par  
  \textsc{Email address:} \texttt{rajabzadeh@ipm.ir, rajabzadeh.hesam@gmail.com} \par
  \addvspace{\bigskipamount}
  \textsc{P. Safaee,}\par
  \textsc{Institut für Mathematik Universität Zürich, Winterthurstrasse 190,} \par{
  8057 Zürich, Switzerland.} \par
  \textsc{Email address:} 
  \texttt{pedram.safaee@math.uzh.ch}, \texttt{pedram.safaee95@gmail.com} \par
}}

\date{September 2023}

\begin{document}

\maketitle
\begin{abstract}
    We prove the positivity of the top Lyapunov exponent of the twisted (spectral) cocycle, associated with IETs, with respect to a family of natural invariant measures. The proof relies on relating the top exponent to limits of exponents along families of affine invariant submanifolds of genus tending to infinity. Applications include an observation about a conjecture of Kontsevich and Zorich, a discrepancy estimate, and a formula for the lower local dimension of spectral measures. 
\end{abstract}

\tableofcontents{}

\section{Introduction}\label{intro}
An interval exchange transformation (IET) is a map $T:[0,1] \to [0,1]$ of the interval to itself for which the interval admits a finite partition by subintervals over each of which $T$ acts as translation in such a way that the image of the interior of the intervals is disjoint from one another. These family of maps arise naturally as the return maps of directional flows on translation surfaces.

The ergodic theory of IETs and translation flows has been studied extensively by many authors (see for instance \cite{KatokStepinWeak, KeaneMinimal, Katoknomixing, MasurIETMeasuredFoliations, VeechGaussmeasure, Zorich97deviation, ForniDeviation, AvilaForniweak, AthreyaForni, Chaika-Annals, Hubert-Ferenczi-2019}). Some of the main highlights of this theory 
include the typicality of properties such as minimality, unique ergodicity, rigidity, and quantitative weak mixing among these families of dynamical 
systems. The main method to establish the mentioned properties is the study 
of the associated renormalization dynamics. To this end, one generally uses associated linear cocycles over the renormalization operators and recasts the properties of individual systems in terms of those cocycles.

The Kontsevich-Zorich cocycle, introduced in the work of Kontsevich and Zorich, is a linear cocycle defined over the Teichmüller flow on the moduli space of abelian differentials. A discrete-time version of this cocycle had been previously introduced by 
Zorich, in an attempt to understand, among other things, the growth rates of Birkhoff sums of observables over typical IETs (see \cite{Zorich97deviation}). The Zorich cocycle is an accelerated version of the Rauzy-Veech cocycle (see \cite{rauzy1979echanges, VeechGaussmeasure}). Study of the Lyapunov spectra of these cocycles has remarkable implications for the dynamics of typical IETs and translation flows. It is well-known 
the cocycles mentioned above preserve a symplectic structure, thereby having symmetric spectra with respect to zero. The Kontsevich-Zorich conjecture asserts that the Lyapunov spectra of the restricted Kontsevich-Zorich cocycle, and consequently that of the restricted Zorich cocycle, are simple and nonzero \cite{kontsevich-Zorich}. This conjecture also provides precise asymptotics for Birkhoff sums in terms of the Lyapunov spectrum (see \cite[Theorem 9.6]{ForniDeviation}). While Giovanni Forni provided, along with many other important things, a partial 
resolution of this conjecture \cite{ForniDeviation}, a complete proof of it, using different methods, was given by Avila and Viana in \cite{AvilaVianaSimplicity}. For the mathematical definitions, see Section \ref{sec: Prelim} and the references therein.

In this paper, we study the spectrum of a linear cocycle introduced by Bufetov and Solomyak, known as the twisted (or spectral) cocycle. This cocycle is designed so that its iterations control the exponential Birkhoff sums of locally constant functions, analogous to the fact that the iterations of the Zorich cocycle govern the behavior of ordinary Birkhoff sums of such functions. This cocycle and its cousins have appeared in several papers studying rates of weak mixing and the spectral measures associated with IETs and substitution dynamical systems (see for instance \cite{BufetovSolomyakHoldergenustwo, ForniTwisted, BufetovSolomyak-HolderRegularity, BufetovSolomyak-selfsimilarity, marshallmaldonado2020modulus, marshallmaldonado2022lyapunov, AvilaForniSafaee}).

To put things in perspective, let $d$ be a positive integer and $\Delta:=\{\lambda \in \mathbb{R}_{+}^d: \|\lambda\|_1=1 \}$, where $\|.\|_{1}$ denotes the $L^1$-norm, be the simplex that parameterizes the length vectors of $d$-IETs. Let $\pi$ be an irreducible permutation on $d$-symbols and $\mathfrak{R}$ its corresponding Rauzy diagram. The Zorich transformation (renormalization) denoted by $\mathcal{R}_{Z}: \Delta \times \mathfrak{R}\to  \Delta \times \mathfrak{R}$ preserves an invariant probability measure, which we call $\mu_Z$. We denote the Zorich cocycle defined over the Zorich transformation by $(\mathcal{R}_{Z}, B^{Z})$. For every $(\lambda, \pi)$, $B^{Z}(\lambda, \pi)$, being an element of $ SL(d, \mathbb{Z})$, induces an automorphism of $\mathbb{T}^d$ which brings us to the definition of the toral Zorich cocycle $\mathcal{R}^{Z}_{\mathbb{T}^d}: \Delta \times \mathfrak{R} \times \mathbb{T}^d \to \Delta \times \mathfrak{R} \times \mathbb{T}^d$ by
\begin{equation}
    \toral^Z(\lambda,\pi, \zeta):= (\mathcal{R}_{Z}(\lambda, \pi), B^Z(\lambda, \pi)\zeta).
\end{equation} 

The main object of study in this paper is the twisted Zorich cocycle, which is a matrix cocycle over $\toral^Z$. Similar to the definition of (ordinary) Zorich cocycle, we associate to every $(\lambda,\pi,\zeta)\in \Delta\times \mathfrak{R}\times \mathbb{T}^d$, a matrix $\tw^{Z}(\lambda, \pi, \zeta)\in GL(d,\mathbb{C})$ with determinant of unit absolute value. In case $\zeta= 0 \in \mathbb{R}^d/\mathbb{Z}^d$, the definition of $\tw^Z(\lambda,\pi,\zeta)$ reduces to that of the ordinary Zorich matrix $B^Z(\lambda,\pi)$ (see Definition \ref{def:twisted-cocycle}). Our main theorem concerns the positivity of the top Lyapunov exponent of this cocycle with respect to a pair of natural invariant probability measures of the base dynamics $\toral^{Z}$, which project onto $\mu_Z$. 
One is the product measure  $\mu_Z \times m_{\mathbb{T}^d}$, where $m_{\mathbb{T}^d}$ is the normalized Lebesgue measure on $\mathbb{T}^d$. The other one is $\sum_{\pi \in \mathfrak{R}}\mu_{Z}\big\vert_{\Delta_{\pi}} \times m_{\mathbb{T}^{2g}_{\pi}}$ whose disintegration along the fibers is supported on $\mathbb{T}^{2g}_{\pi}$ which are $2g$-dimensional subtori of $\mathbb{T}^d$ (see Section \ref{sec: Twisted Cocycle}). We now turn to the statement of our
\begin{theoremain}\label{main-thm}
Let $\pi$, $\mathfrak{R}$, and $\mu_{Z}$ be as before. If $g>1$, then the top Lyapunov exponent of the twisted cocycle is positive with respect to both invariant measures $\mu_{Z} \times m_{\mathbb{T}^d}$ and $\sum_{\pi \in \mathfrak{R}}\mu_{Z}\big\vert_{\Delta_{\pi}} \times m_{\mathbb{T}^{2g}_{\pi}}$.
\end{theoremain}
While the similar statement for the ordinary Zorich cocycle is trivial, this result has remained unknown since the introduction of the twisted cocycle in the early 2010s. It is worthwhile to mention that our proof yields an explicit lower bound for the top exponent in terms of the combinatorial data. 
We remark that in the case of genus $1$, all exponents turn out to be zero. It is also true that in higher genus the cocycle admits at least two zero exponents. The proofs of these claims, which will appear in an upcoming paper by the authors, are of a different flavor. On par with the Kontsevich-Zorich conjecture, the following question seems most natural.  
\begin{question}
    Determine whether the twisted cocycle admits a $2g-2$-dimensional invariant subbundle over which the Lyapunov spectrum is simple and nonzero.
\end{question}

The interest of the mathematical community in the nondegeneracy of the spectrum of cocycles goes far beyond the special case of cocycles over the space of interval exchange transformations, 
and the Teichmuller flow (see \cite{Herman,Wilkinson-What-are-Lyapunov,Brown-Fisher-Hurtado, Avila2023kam} and references therein). Other prominent examples include the invariance principle \cite{Furstenberg, Ledrappier,BonattiVianaGomezMont,invariance-principle, AvilaVianaSantamaria}, the genericity of nonvanishing of Lyapunov spectrum for cocycles over sufficiently chaotic base dynamics \cite{VianaNonvanishing_of_exponents}, and the 
positivity of the Lyapunov exponent for certain dynamically defined Schrödinger operators \cite{Avila-Damanik-Zhang}. While the main pillar of the aforementioned papers is a dexterous application of the invariance principle, we 
use here a different approach that may be applicable in broader contexts.

The paper is organized as follows.  Section \ref{sec: Prelim} is comprised of preliminary definitions and notations regarding interval exchange transformations, translation surfaces, renormalization schemes, and some basics of cocycle dynamics. Section \ref{sec: Twisted Cocycle} deals with toral and twisted 
cocycles. In particular, it defines these cocycles, and provides a multifaceted study of certain invariant measures of the toral cocycle and their ergodic properties. Moreover, a control for 
the growth of exponential Birkhoff sums of locally constant functions is presented in terms of the growth of the corresponding vectors under the twisted cocycle. Section \ref{sec: IETs with positive exponent} is 
devoted to the construction of some tools for the proof of Main Theorem presented in Section \ref{sec:proof-main-theorm}. The idea of the proof 
is to approximate $\mu_Z\times m_{\mathbb{T}^d}$ by measures $\mu_Z\times \nu_p$ with atomic disintegration along the fibers and using semi-continuity of the top Lyapunov exponent with respect to $\toral^Z$ invariant measures. Then, the main difficulty is to obtain a uniform lower bound for the top exponent of the above-mentioned measures. A sketch of the proof of this fact is provided below. 

Given a prime number $p$ and a vector $\n \in \mathbb{Z}^d$, to simplify the study of twisted Birkhoff sums with respect to the twist parameters $\zeta_k=\frac{k}{p} \n$, an untwisting mechanism 
is devised that allows to realize the twisted Birkhoff sums with twist factors $\zeta_k$ of locally constant functions along a subsequence of times as ordinary Birkhoff sums for a suspended system $\spn(\lambda, \pi)$ (see Definition \ref{Stopping time}, and Lemma \ref{Small kernel surjection}). A translation surface $X$ is constructed in such a way that $\spn(\lambda, \pi)$'s are return maps to a fixed transversal of translation flows on $X$ (see Subsection \ref{translation-surface}). An analysis of the genus and order of singularities of $X$ 
(see Lemma \ref{genus bound}, and Lemma \ref{Order of singularities}), along with an application of \cite[Theorem 1]{Eskin-Kontsevich-Zorich} to the $SL(2, \mathbb{R})$ invariant closure of $X$, yields a uniform (independent of $p$) lower bound 
for the mean of the Lyapunov exponents of the Kontsevich-Zorich cocycle. This fact, combined with the comparison of the growth rate of exponential Birkhoff sums of locally constant functions with the 
growth rate of the corresponding vectors under the twisted cocycle (see Proposition \ref{exponent-comparison}), provides a uniform lower bound for the pointwise exponent of the cocycle for a fixed portion of the points $\zeta_1, \ldots, \zeta_{p-1}$ (see Proposition \ref{Sum of exponents} for more details). This proves the uniformity of the lower bound of the top exponent of the twisted cocycle with respect to $\mu_{Z}\times \nu_p$.

Section \ref{sec: Applications} collects various applications of the main result, and the ideas developed in the previous sections. First, a surprising connection to one of the conjectures of Kontsevich-Zorich regarding high genus limits of second Lyapunov exponents (see Theorem \ref{Kontsevitch-Zorich-Conjecure}) is given. Second, the discrepancy of ergodic sums under IETs modulo prime numbers is defined and analyzed. Last, an application to the dimension of spectral measures is provided. 

\section*{Acknowledgements} We would like to express our deep gratitude to Giovanni Forni for suggesting this problem, his immense support and encouragement as well as many fruitful discussions during work on this project. We also thank Artur Avila and Carlos Matheus for helpful conversations.

\section{Preliminaries} \label{sec: Prelim}
We collect here a number of definitions and tools that we need in the forthcoming sections.
\subsection{Basics of cocycle dynamics}\label{basic-ergodic-theory}

Let $X$ be a separable second countable metric space. We fix the Borel $\sigma$-algebra on $X$ and measurable functions are defined with respect to this $\sigma$-algebra. For a measurable transformation $f:X\to X$, a point $x\in X$ is called {\it Birkhoff regular}, whenever for every compactly supported continuous observable $\varphi:X\to \mathbb{R}$, the following limit exists 
\begin{equation}\label{eq:Birkhoff-erg-thm}
\lim\limits_{n\to +\infty} \frac{1}{n}\sum\limits_{j=1}^n\varphi(f^j(x))   
\end{equation}
The Birkhoff ergodic theorem guarantees that the set of Birkhoff regular points has full measure with respect to all the $f$-invariant probability measures on $X$. Moreover, if the invariant measure $\mu$ is ergodic for $f$, then for every $\varphi\in L^1(X,\mu)$, the limit in \eqref{eq:Birkhoff-erg-thm} equals $\int_X \varphi {d}\mu$ for $\mu$-almost every $x\in X$. 

A linear cocycle over the base dynamics $f:X\to X$, is a map $F:X\times V\to X\times V $ of the form 
\begin{equation}
    F(x,v)=\big(f(x), A(x)v)
\end{equation}
where $V$ is a normed vector space and $A:X\to GL(V)$ is measurable. 

\begin{thm}(Oseledets multiplicative ergodic theorem) Let $f:X\to X$ be a measurable map, preserving invariant probability measure $\mu$, and $F:X\times V\to X\times V$ be a linear cocycle, as above, such that $\max\{\log\|A(x)\|,0\}\in L^1(X,\mu)$.\\ Then, there exist measurable functions $k=k(x)\in\mathbb{N}$, $\vartheta_k(x)<\vartheta_{k-1}<\cdots<\vartheta_1$, and linear subspaces 
\begin{equation}
\{0\}\subsetneq V_k(x)\subsetneq V_{k-1}(x)\subsetneq\cdots \subsetneq V_1(x)\subset V,    
\end{equation}
defined for $\mu$-almost every $x\in X$ and depending measurably on $x\in X$, satisfying the following properties 
\begin{itemize}
    \item[i)] For $\mu$-almost every $x\in X$, $k(x)=k(f(x))$ and for every $i\in \{1,2,\ldots,k\}$, $\vartheta_i(x)=\vartheta_i(f(x))$.
    \item[ii)] For $\mu$-almost every $x\in X$, and every $i\in \{1,2,\ldots,k\}$, $A(x)V_i(x)=V_i(f(x))$.
    \item[iii)] For every $v\in V_{i}(x)\setminus V_{i+1}(x)$, 
    \begin{equation}
        \lim\limits_{n\to +\infty}\frac{1}{n}\log\|A_n(x)v\|=\vartheta_i(x),
    \end{equation}
where, $A_n(x)=A(f^{n-1}(x))\cdots A(f(x))A(x)$. 
\end{itemize}
\end{thm}

For the measure $\mu$, the functions 
$\vartheta_1(x),\ldots,\vartheta_k(x)$ in the Oseledets theorem are known as the {\it Lyapunov exponents} of the cocycle (with respect 
to $\mu$) and the set of points $x\in X$ for which items (i),(ii),(iii) of the theorem hold true, are called {\it Oseledets regular} with respect to $\mu$. For general background on ergodic theory, Lyapunov exponents, and cocycles see \cite{Viana-Oliveira, VianaLyapunovExponentsBook}.

\subsection{IETs and translation surfaces}

An \textit{interval exchange transformation} (IET) is a dynamical system that acts on an interval by reshuffling the elements of a 
finite partition of the interval into subintervals. More precisely, let $I=[0,1)$, $d>1$ an integer, $\mathcal{A}$ a set with $|\mathcal{A}|=d$, and $\pi:=(\pi_t, \pi_b)$ where $\pi_t, \pi_b: \mathcal{A} \to \{1, \ldots, d\}$ are bijections. Let $\Delta^{A}:= 
\{ \lambda \in \mathbb{R}_{+}^\mathcal{A}: \sum_{\alpha \in \mathcal{A}} \lambda_\alpha=1\} \subset \mathbb{R}_{+}^{\mathcal{A}}$ (or simply $\Delta$ when there is no 
ambiguity). Let \[I_{\alpha_{0}}:=
\Big[\sum\limits_{\pi_t(\alpha)<\pi_{t}(\alpha_0)}\lambda_{\alpha}, \sum\limits_{\pi_t(\alpha) \leq \pi_t(\alpha_0)} \lambda_\alpha\Big),\; \forall \alpha_0 \in \mathcal{A},\] where $\lambda \in \Delta^{\mathcal{A}}$. Then the interval exchange transformation 
associated with $(\lambda, \pi)$ is the map $T_{\lambda, \pi}: I \to I$ defined by 
\[T_{\lambda, \pi}(x)=x-\sum\limits_{\pi_t(\alpha)<\pi_t(\alpha_0)}\lambda_{\alpha}+\sum\limits_{\pi_b(\alpha)
<\pi_b(\alpha_0)}\lambda_{\alpha}, \quad x\in I_{\alpha_0},\]
where $\alpha_0 \in \mathcal{A}$. It is clear from the definition 
that there exists $\delta \in \mathbb{R}^{\mathcal{A}}$ such that  $T_{\lambda, \pi}(x)=x+\delta_\alpha$ for every $x\in I_\alpha$.

Let $X$ be a Riemann surface and $\Sigma \subset X$ be a finite set. A \textit{translation structure} $\omega$ on $(X, \Sigma)$ is an abelian differential $\omega$ with singularities in $\Sigma$. That is, for every $p \in X$ there exist an open set $U_p$ containing $p$ and holomorphic local coordinates $z: U_p \to \mathbb{C}$ such that
\begin{equation}
    \omega = z^{m(p)} dz,
\end{equation}
where $m=0$ for all $p \in X \setminus \Sigma$. A point $p\in \Sigma$ is called a {\it singularity} of order $m=m(p)$. In case $m(p)=0$ for $p\in \Sigma$, we call $p$ a {\it fake singularity}. The pair $(X, \omega)$ is also called a {\it translation surface}. $\omega$ induces a flat metric with singularities at $\Sigma$.  Gauss-Bonnet implies that
\begin{equation}
    2g(X)-2= \sum_{i=1}^{\kappa} m_i.
\end{equation}
A translation surface is an equivalence class of polygons in $\mathbb{C}$: Every translation surface admits a finite collection of polygons, together with a choice of pairing of parallel edges of equal length that are on opposite sides. The moduli space of all translation structures with singularities of orders $m_1, \ldots, m_\kappa$ is denoted by $\mathcal{H}(m_1, \ldots, m_\kappa)$. This space admits a natural action of $SL(2, \mathbb{R})$: for every $g\in SL(2, \mathbb{R})$, $g(X, \omega)$ is understood as the linear action of the matrix $g$ on the union of polygons representing $(X, \omega)$. 

The \textit{Kontsevich-Zorich cocycle} is a monodromy cocycle defined over the action of $SL(2, \mathbb{R})$ on the moduli space of translation surfaces. In the ensuing sections, we will use a discrete-time version of this cocycle (the Zorich cocycle), defined below. For further background on translation surfaces, their relations to IETs, their moduli spaces, and the Kontsevich-Zorich cocycle, we refer the reader to \cite{zorich2006flat, Forni-Matheus, Wright-translation-surface, yoccoz2007interval, viana2008dynamics}.

\subsection{Rauzy-Veech and Zorich renormalizations}
\smallskip
Let $T_{\lambda^{(0)}, \pi^{(0)}}$ be an IET on the interval $I^{(0)}$ with parameters $(\lambda^{0}, \pi^{0})$
and $\alpha_{t}, \alpha_{b}$ be the letters with the property that $\pi_{t}(\alpha_{t})=\pi_{b}(\alpha
_{b})=d$. Assuming that $\lambda^{(0)}_{\alpha_t} \neq \lambda^{(0)}_{\alpha_b}$, we define the Rauzy induction algorithm by considering the first return map of $T_{\lambda^{(0)}, \pi^{(0)}}$ to the subinterval $[0, 1- \min \{\lambda_{\alpha_t}, \lambda_{\alpha_b}\})$. The resulting transformation is the IET $T_{\lambda^{(1)}, \pi^{(1)}}$ defined as follows
\begin{itemize} 
\item[i)] $\lambda^{(0)}_{\alpha_{t}}> \lambda^{(0)}_{\alpha_{b}}$: then let $\lambda^{(1)}_{\alpha}:=\lambda^{(0)}_{\alpha}$ for all $\alpha \neq \alpha_{t}$, $\lambda^{(1)}_{\alpha_{t}}:= \lambda^{(0)}_{\alpha_{t}}- \lambda^{(0)}_{\alpha_{b}}$, $\pi_{t}^{(1)}:=\pi_{t}^{(0)}$, and
\begin{equation} \label{eq:top}
(\pi^{(1)}_{b})^{-1}(i):=
\begin{cases}
(\pi^{(0)}_{b})^{-1}(i) & \text{if}\; i \leq \pi^{(0)}_{b}(\alpha_t),\\
\alpha_{b} & \text{if}\;
i=\pi^{(0)}_{b}(\alpha_t)+1,\\
(\pi^{(0)}_{b})^{-1}(i-1) & \text{otherwise},
\end{cases}
\end{equation}
in which case $(\lambda^{(0)}, \pi^{(0)})$ is said to be of \textit{top type}
(since $\alpha_t$ ``wins'' against $\alpha_b$).
\item[ii)]$\lambda_{\alpha_{b}}^{(0)}>\lambda_{\alpha_{t}}^{(0)}$: then let $\lambda_{\alpha}^{(1)}:=\lambda_{\alpha}^{(0)}$ for all $\alpha \neq \alpha_{b}$, $\lambda_{\alpha_{b}}^{(1)}:=\lambda_{\alpha_{b}}^{(0)}- \lambda_{\alpha_{t}}^{(0)}$, $\pi_{b}^{(1)}:=\pi_{b}^{(0)}$, and 
\begin{equation}\label{eq:bottom}
(\pi^{(1)}_{t})^{-1}(i):=
\begin{cases}
(\pi^{(0)}_{t})^{-1}(i) & \text{if}\; i \leq \pi^{(0)}_{t}(\alpha_b),\\ 
\alpha_{t} & \text{if}\; i=\pi^{(0)}_{t}(\alpha_b)+1,\\
(\pi^{(0)}_{t})^{-1}(i-1) & \text{otherwise},
\end{cases}
\end{equation}
in which case $(\lambda^{(0)}, \pi^{(0)})$ is said to be of \textit{bottom type}
(since $\alpha_b$ ``wins'' against $\alpha_t$).
\end{itemize}
The combinatorial data $\pi$ is \textit{irreducible}, 
if there exists no $1\leq k < d$ such that 
$\pi_{t}^{-1}(\{1,\dots,k\})=\pi_{b}^{-1}(\{1,\dots,k\})$. We denote the space of such combinatorial data by $\mathfrak{S}^{0}(\mathcal{A})$. The above operations define an equivalence relation that partitions
the set of combinatorial data into equivalence \textit{Rauzy classes}, denoted by $\mathfrak{R} \subset \mathfrak{S}^{0}(\mathcal{A})$. 

Let $\mathfrak{R}$ be a Rauzy class and $\pi \in \mathfrak{R}$ a permutation. We define $\mathcal{Q}_{R}: \mathbb{R}_{+}^{\mathcal{A}} \times \mathfrak{R} \to 
\mathbb{R}_{+}^{\mathcal{A}} \times \mathfrak{R}$, called \textit{Rauzy 
induction map}, to be the map that sends $(\lambda^{(0)}, \pi^{(0)})$ to $(\lambda^{(1)}, \pi^{(1)})$, according to equations $\ref{eq:top}$ and $\ref{eq:bottom}$. Rescaling $\lambda^{(1)}$ back to size $1$ 
(size being the $\ell^{1}$-norm) yields a transformation $\mathcal{R}_{R}: 
\Delta^{\mathcal{A}} \times \mathfrak{R} \to \Delta^{\mathcal{A}} \times \mathfrak{R}$, which we call the \textit{Rauzy renormalization map}.

Zorich defined an accelerated version of the Rauzy renormalization, which is now called \textit{Zorich transformation}.

\begin{definition}
For an element $(\lambda, \pi) \in \Delta^{\mathcal{A}} \times \mathfrak{R}$, let    $n:=n(\lambda, \pi)$ be the smallest 
 $n\in \mathbb{N}$ for which the type of $(\lambda^{(n)}, \pi^{(n)})=\mathcal{Q}_{R}^{n}(\lambda, \pi)$ is different from that of $(\lambda, \pi)$. Then, the Zorich induction, $\mathcal{Q}_{Z}(\lambda, \pi)$ is defined by 
\[\mathcal{Q}_{Z}(\lambda, \pi):=(\lambda^{(n)}, \pi^{(n)})\] 
and the Zorich renormalization is defined by rescaling the intervals back to size $1$
\[\mathcal{R}_Z(\lambda,\pi):=\big(\frac{\lambda^{(n)}}{|\lambda^{(n)}|},\pi^{(n)}\big).\]
\end{definition}
Zorich showed the following
\begin{thm} (\cite{ZorichGaussmeasure}) Let $\mathfrak{R} \subset \mathfrak{S}^{0}(\mathcal{A})$ be a Rauzy class. Then $\mathcal{R}_{Z}|_{\Delta^{\mathcal{A}} \times \mathfrak{R}}$ admits a unique ergodic invariant probability measure $\mu_{Z}$ equivalent to the Lebesgue measure.
\end{thm}

\subsection{Rauzy-Veech and Zorich cocycles}
The {\it Rauzy-Veech cocycle} is a matrix cocycle that stores the visitation data of the renormalization procedure. In other words, for $(\lambda,\pi)$, $B^{R}(\lambda, \pi)$ is the $d\times d$ matrix  given by the following formula
\begin{equation}
    B^{R}(\lambda, \pi):=
    \begin{cases}
    I+ E_{\alpha_{b}\alpha_{t}} & \text{if} \; (\lambda, \pi)\; \text{is of top type,}\\
    I+ E_{\alpha_{t}\alpha_{b}} & \text{if} \; (\lambda, \pi) \; \text{is of bottom type,}
    \end{cases}
\end{equation}
where for all $i,j \in \{1, \dots, d\}$, the symbol $E_{ij}$ denotes the $d\times d$ elementary matrix with a single non-zero entry equal to $1$ in position $(i,j)$.

With the above definition, $(\mathcal{R}_{R}, B^{R})$ forms an integral
cocycle. The corresponding cocycle over the Zorich transformation is called the {\it Zorich cocycle} and is denoted by $(\mathcal{R}_{Z},B^{Z})$. 

Let us consider a diagram with vertices the set of all combinatorial data in the Rauzy class $\mathfrak{R}$. We connect two vertices with
an arrow labeled by $t$ or $b$ specifying the type of operation needed 
to go from the initial vertex of the edge to its terminal vertex. We denote the space of all paths in this diagram by $\Pi(\mathfrak{R})$. Now, if $\gamma \in \Pi(\mathfrak{R})$ is a path obtained by concatenation of labeled edges $\gamma_1, \gamma_2,\dots, \gamma_k$, for some $k\in \mathbb{N}$, we let $B_{\gamma}$ be the product of the corresponding renormalization matrices. That is,
\begin{equation}
    B_{\gamma}:=B_{\gamma_k}\cdots B_{\gamma_1}.
\end{equation}
Note that if $\lambda^{(\gamma)}$ is the length vector corresponding to the induction of $(\lambda, \pi)$ following the path $\gamma$ (we are tacitly assuming that $\gamma$ starts at $\pi$) then 
\begin{equation}\label{lambda parameters under renormalization}
    \lambda^{(\gamma)}B_{\gamma}=\lambda.
\end{equation}
where we are viewing $\lambda^{(\gamma)}$ and $\lambda$ as row vectors.

\subsection{Veech's zippered rectangles}
\label{Veech's Zippered Rectangles}

Veech introduced a construction, called \textit{zippered rectangles}, for creating a translation surface out of an interval exchange transformation and a locally constant roof function. To be more precise, to each permutation $\pi \in \mathfrak{R}$ there corresponds a $2g$-dimensional subspace $H(\pi) \subset \mathbb{R}^d$. Given $\lambda \in \Delta$ and $h \in H^{+}(\pi)$, the positive cone consisting of vectors with positive coordinates in $H(\pi)$, Veech's method allows to construct a family of translation surfaces each corresponding to $\lambda, \pi, h$ and a certain zipping data. As long as one is concerned with the dynamics of the translation flow in the north direction, the zipping data do not play any role, and thus in this paper, we omit any considerations thereof. Strictly speaking, the flow in the north direction is measurably isomorphic to the special (suspension) flow with roof function $h$ over the IET $T_{\lambda, \pi}$. We give the details of Veech's construction below.

Let $\Omega_\pi$ be the following transformation

\begin{equation}
    (\Omega_{\pi})_{(\alpha, \beta)}:= 
    \begin{cases}
    +1 & \text{if} \; \pi_{b}(\beta)< \pi_b(\alpha), \; \pi_t(\beta)> \pi_t(\alpha),\\
    -1 & \text{if} \; \pi_b(\beta)> \pi_b(\alpha), \; \pi_t(\beta)< \pi_t(\alpha),\\
    0 & \text{otherwise},
    
    \end{cases}
\end{equation}
and  
\begin{equation}
    T^{+}(\pi):=\bigg\{\tau \in \mathbb{R}^{d}: \sum_{\pi_{t}(\alpha)< \pi_{t}(\beta)} \tau_{\alpha}>0, \; \sum_{\pi_b(\alpha)> \pi_{b}(\beta)} \tau_{\alpha}<0\bigg\}.
\end{equation}
Take $H^{+}(\pi) \subset \mathbb{R}_{+}^{d}$ to be the image of $T^{+}(\pi)$ under $-\Omega_{\pi}$.

Let us call $\tau \in T^{+}(\pi)$ the zipping data, and $h= -\Omega_{\pi}(\tau) \in H^{+}(\pi)$ the roof function. The translation surface corresponding to $(\lambda, \pi, h, \tau)$ is obtained by applying certain identifications on the union of rectangles of the 
form $I_{\alpha}\times [0,h_\alpha]$. Roughly speaking, the zipping data $\tau$ determine the heights up to which two adjacent rectangles are sewn together in the vertical direction. Regarding the horizontal sides, we identify the top edge of the rectangle $I_\alpha \times [0, h_\alpha]$ with the subinterval of $T_{\lambda, \pi}(I_\alpha) \subset 
I$. The resulting surface $M=M(\lambda, \pi, \tau, h)$ is a translation surface whose $1$-st cohomology space $H_1(M)$ corresponds to the subspace $H(\pi)$. Then $\Omega(\pi)$ is the symplectic intersection 
pairing form on $H_1(M)$, which is well known to be preserved under the Zorich cocycle. This implies that the spectrum of the Zorich cocycle restricted to the subspace $H(\pi)$ is symmetric with respect 
to zero. For a detailed explanation of this construction and its consequences, we refer the reader to \cite{VeechGaussmeasure}, or \cite{viana2008dynamics}.

Let $g$ and $\kappa$ be the genus and the number of singularities of the surface $M$, respectively. Denote by $N(\pi)$ the kernel of $\Omega(\pi)$. The following are due to Veech \cite{VeechGaussmeasure}
\begin{equation} \label{genus and singularities}
    \dim(N(\pi))= \kappa -1, \; d= 2g+ \kappa-1.
\end{equation}

Veech \cite{VeechGaussmeasure} showed that the Rauzy induction extends to an invertible dynamical system 
(natural extension), called {\it Rauzy--Veech induction}, on the space of zippered rectangles. Moreover, the locally constant subbundles $H(\pi), N(\pi)$ are respectively invariant and contravariant under the {\it Rauzy--Veech cocycle}. That is, for $\mathcal{R}_{R}(\lambda, \pi)=(\lambda^{(1)}, \pi^{(1)})$, we have
\begin{align}
    B^{R}(\lambda, \pi)\cdot H(\pi)=H(\pi^{(1)}) \\
    {}^tB^{R}(\lambda, \pi)\cdot N(\pi^{(1)})=N(\pi)\,.
\end{align}
In the same paper, Veech also showed that for any loop $\gamma \in \Pi(\mathfrak{R})$ starting and ending at $\pi$, ${}^tB_\gamma$ acts as identity on $N(\pi)$. Therefore the Lyapunov exponents of the Zorich cocycle corresponding to the complement of $H(\pi)$ in $\mathbb{R}^d$ are all zeros. However, Forni \cite{ForniDeviation} showed that the Zorich cocycle restricted 
to $H(\pi)$ is non-uniformly hyperbolic, i.e., its Lyapunov exponents $\chi_1(\mathfrak{R}) \geq \chi_2(\mathfrak{R}) \geq \dots \geq \chi_{2g}(\mathfrak{R})$ are all nonzero. This result was later strengthened by Avila and 
Viana in \cite{AvilaVianaSimplicity}, where they showed that the cocycle has simple Lyapunov spectrum, thereby finishing the proof of the so-called Kontsevich-Zorich conjecture (see \cite{kontsevich-Zorich} and Section \ref{intro} of this paper). 

\begin{thm} (\cite{ForniDeviation}, \cite{AvilaVianaSimplicity}) For any Rauzy class $\mathfrak{R} \subset \mathfrak{S}_d$ the Zorich cocycle on $\Delta \times \mathfrak{R}$ is non-uniformly hyperbolic and has simple Lyapunov spectrum. Therefore,
\begin{equation}
    \chi_1(\mathfrak{R})>\dots>\chi_g
    (\mathfrak{R})>0> -\chi_{g}(\mathfrak{R})>\dots>-\chi_{1}(\mathfrak{R}).
\end{equation}
The symmetry of the Lyapunov spectrum is due to the cocycle being symplectic. 
\end{thm}

\section{Toral and twisted cocycles} \label{sec: Twisted Cocycle}
The Rauzy-Veech and Zorich cocycles naturally induce cocycles acting by linear automorphisms on the $d$-torus $\mathbb{T}^d= \mathbb{R}^d/\mathbb{Z}^{d}$, which we call the \textit{toral Rauzy--Veech cocycle} and the \textit{toral Zorich cocycle}, respectively. 
\begin{definition} The toral Rauzy-Veech cocycle is defined as 
\begin{align}
 \toral^{R}:&\Delta \times \mathfrak{R} \times \mathbb{T}^{d} \longrightarrow  \Delta \times \mathfrak{R} \times \mathbb{T}^{d},
\\
& (\lambda,\pi, \zeta)\longmapsto (\mathcal{R}_{R}(\lambda, \pi), B^R(\lambda, \pi)\zeta).
\end{align}
Similarly, we call the accelerated version of the above cocycle the toral Zorich cocycle which we define as \[\toral^Z(\lambda,\pi, \zeta):= (\mathcal{R}_{Z}(\lambda, \pi), B^Z(\lambda, \pi)\zeta).\] 
\end{definition}

The twisted cocycles are cocycles defined over the toral Rauzy-Veech and Zorich cocycles as follows.

\begin{definition} \label{def:twisted-cocycle}(Twisted Rauzy-Veech and Zorich cocycles) For an element $(\lambda, \pi, \zeta) \in \Delta \times \mathfrak{R} \times \mathbb{T}^{d}$, we define the corresponding twisted Rauzy-Veech matrix by
\begin{equation} 
    \tw^{R}(\lambda, \pi, \zeta):=
    \begin{cases}
    I+ \exp(2\pi i \zeta_{\alpha_b})E_{\alpha_{b}\alpha_{t}}& \text{if} \; (\lambda, \pi)\; \text{is of top type}\\
    I+ E_{\alpha_{t}\alpha_{b}} +(-1+ \exp(2\pi i 
    \zeta_{{\alpha}_b})) E_{\alpha_t, \alpha_t} & \text{if} \; (\lambda, \pi) \; \text{is of bottom type.}
    \end{cases}
\end{equation}
The corresponding twisted Zorich matrix is defined by
\begin{equation}
    \tw^{Z}(\lambda, \pi, \zeta):= \prod\limits_{k=0}^{n(\lambda, \pi)-1} \tw^{R}((\toral^{R})^{j}(\lambda, \pi, \zeta)),
\end{equation}
and the corresponding twisted Zorich cocycle is given by
\begin{align*}
    \twc:&~\Delta\times \mathfrak{R} \times \mathbb{T}^{d}\times \mathbb{C}^{d} \to \Delta\times \mathfrak{R} \times \mathbb{T}^{d}\times \mathbb{C}^{d}\\
    &(\lambda, \pi, \zeta, f)\longmapsto(\toral(\lambda, \pi, \zeta), \tw^{Z}(\lambda, \pi, \zeta) f)
\end{align*}
\end{definition}
\begin{rmk}
    Notice that the definition of the twisted Rauzy-Veech matrix depends solely upon the equivalence class of $\zeta$ mod $\mathbb{Z}^{d}$.
\end{rmk}

\subsection{Invariant measures}
For $p>1$ prime, let $Q_p \subset \mathbb{T}^d$ be the set of rational points of the form $\{\n/p \vert \n \in \mathbb{Z}^d\setminus \{0\}\} \subset \mathbb{R}^d/\mathbb{Z}^d$. The subspace $H(\pi)$ is defined over $\mathbb{Q}$ and therefore $H(\pi) \cap \mathbb{Z}^d \subset H(\pi)$ is a lattice. We denote by $\mathbb{T}^{2g}_{\pi}$ the torus obtained via quotienting $H(\pi)$ by this lattice. We let $Q^{(\pi)}_{p} \subset H(\pi)/(H(\pi) \cap \mathbb{Z}^d)$ be the set of nonzero rational points with denominator $p$ in $\mathbb{T}^{2g}_\pi$. Then we define $\nu_{p}, \nu^{(\pi)}_{p}$ to be the atomic measures that give equal weight to the elements of $Q_p$ and $Q^{(\pi)}_{p}$, respectively. Let $m_{\mathbb{T}^{2g}_{\pi}}$ be the Lebesgue measure on $\mathbb{T}^{2g}_{\pi}$ normalized to be a probability measure. It can be readily verified that the measures
\begin{itemize}
    \item $\sum_{\pi \in 
    \mathfrak{R}}\mu_{Z}\big\vert_{\Delta_{\pi}} \times m_{\mathbb{T}^{2g}_{\pi}}$, $\mu_Z \times m_{\mathbb{T}^d}$, 
    \item $\sum_{\pi \in 
    \mathfrak{R}}\mu_{Z}\big\vert_{\Delta_{\pi}} \times \nu^{(\pi)}_p$, $\mu_Z \times \nu_p$
\end{itemize}
are invariant under $\toral^{Z}$. 

Let $k$ be a positive integer and $\gamma= \gamma_1 \cdots \gamma_k$, where $\gamma_i$'s are single arrows in the diagram (labeled edges), be a loop 
starting and ending at $\pi$ whose corresponding Rauzy-Veech matrix $B_\gamma$ has 
only positive entries. Moreover, assume that for every $1 \leq j \leq k$ $\gamma_1 \cdots \gamma_j \neq \gamma_{k-j+1} \cdots \gamma_k$.  Let $\Delta_{\gamma}:={}^tB_\gamma. \Delta\subset \Delta$ where ${}^tB_{\gamma}.$ is understood as the projective action on the simplex. Then $\Delta_\gamma$ is the simplex of parameters 
whose first $k$ steps of renormalization follow $\gamma$ (See \eqref{lambda parameters under renormalization}). We denote by $\mathcal{R}^{\gamma}$,  $\mathcal{R}^{\gamma}_{\mathbb{T}^d}$ the return maps of $\mathcal{R}_Z$ and $\mathcal{R}^{Z}_{\mathbb{T}^d}$ to $\Delta_\gamma \times \{\pi\}$ and $\Delta_\gamma \times \{\pi\} \times \mathbb{T}^d$, respectively. We introduce preferred invariant measures for $\mathcal{R}^{\gamma}$ and $\mathcal{R}^{\gamma}_{\mathbb{T}^d}$ which will be used in the statement of Proposition \ref{Exponential mixing}.

Let $\Omega$ be the set of loops ${\ell} \in \Pi(\mathfrak{R})$ starting and ending at $\pi$ with the property that the type of the last arrow in $\ell$ differs from 
that of the starting arrow $\gamma_1$ in $\gamma$. Then there exists a measurable countable partition of $\Delta_\gamma$ into sets with disjoint interior of the form $\Delta^{(\ell)}:= {}^tB_{\gamma {\ell} \gamma}. \Delta$ for ${\ell} \in \Omega$ 
whose elements are mapped onto $\Delta_\gamma$ under $\mathcal{R}^{\gamma}$. Furthermore, for ${\bf l} :=(\ell_0,\ldots, \ell_{n-1}) \in \Omega^n$ we define 
\begin{equation}
    \Delta^{\bf l}:= \big\{ x \in \Delta_\gamma \vert (\mathcal{R}^{\gamma})^{k}(x) \in \Delta ^{(\ell_k)},\; 0\leq k \leq n-1 \big\}
\end{equation}
Positivity of the entries of $B_\gamma$ implies that the inverse branches are all 
projective contractions and thus by \cite[Lemma 2.1]{AvilaForniweak} there exists a 
unique absolutely continuous $\mathcal{R}^{\gamma}$-invariant measure $\mu_\gamma$ on $\Delta_\gamma$ which satisfies the following bounded distortion property. There exists $K>0$ such that for every $m, n \in \mathbb{N}$ and every $\textbf{l} \in \Omega^n, \textbf{l}' \in \Omega^m$
\begin{equation}\label{bounded-distortion}
    \frac{1}{K} \mu_\gamma (\Delta^{\bf l'}) \leq \frac{\mu_\gamma(\Delta ^{{\bf l l}'})}{\mu_\gamma(\Delta ^{\bf l })} \leq K \mu_\gamma (\Delta ^{{\bf l}'}),
\end{equation}
where ${\bf l} {\bf l}'$ stands for the concatenation of $\bf l$ and ${\bf l}'$. We rely on many important theorems of several authors (\cite{AvilaForniweak,AvilaDelecroixweak, AvilaVianaSimplicity}) which are based upon this quasi-independence property. 

Notice that absolute continuity of $\mu_\gamma, \mu_Z$ with respect to Lebesgue implies that $\mu$ is equal to the measure induced by $\mu_Z$ under acceleration of $\mathcal{R}_{Z}$. Then $\mu_\gamma \times m_{\mathbb{T}^{2g}_{\pi}}$ is the $\mathcal{R}^{\gamma}_{\mathbb{T}^d}$-invariant measure on $\Delta_{\gamma} \times \mathbb{T}^{d}$ induced by $\sum_{\pi \in 
    \mathfrak{R}}\mu_{Z}\big\vert_{\Delta_{\pi}} \times m_{\mathbb{T}^{2g}_{\pi}}$.

\begin{prop} \label{Exponential mixing}With the same notations as above, $(\mathcal{R}^{\gamma}_{\mathbb{T}^d}, \mu_\gamma \times m_{\mathbb{T}^{2g}_{\pi}} )$ is exponentially mixing. 
\end{prop}

\begin{proof}

    Let $f(x,\zeta):= \sum_{\n\in \mathbb{Z}^{d}} f_{\n}(x) \exp(2\pi i \n. \zeta), h(x, \zeta):= \sum_{\n\in \mathbb{Z}^{d}} h_{\n}(x) \exp(2\pi i \n. \zeta)$ be two observables with the property that
       \begin{equation}\label{regularity:f,h}
        \|h_{\n}\|_{\infty} \leq C_h \|\n\|^{-\beta}, \;\; \sum_{\n \in \mathbb{Z}} \|f_{\n}\|_2 < \infty.
    \end{equation}
    Then
    \begin{align} \label{exponential mixing ineq}
        \int f \circ (\mathcal{R}^{\gamma}_{\mathbb{T}^d})^k(x, \zeta) \overline{h(x, \zeta)} d\mu(x) dm_{\mathbb{T}_{\pi}^2g}(\zeta) &=\sum 
        _{\n \in \mathbb{Z}^d} \int f_{\n}\circ(\mathcal{R}^{\gamma})^{k}(x)\overline{h_{B^{t}_k(x). \n}(x)}d\mu(x)\\
        & \leq C \sum_{\n \in \mathbb{Z}^d} \bigg( \|f_{\n} \|_{2} \sum_{{\bf l} \in \Omega^{k}} \| h_{{}^tB^{{\bf l}}.\n} \|_{\infty} \mu_\gamma( \Delta^{{\bf l}}) \bigg) \\
        & \leq CC_h \sum_{\n \in \mathbb{Z}^d} \bigg( \|f_{\n} \|_{2}  \sum_{{\bf l} \in \Omega^{k}} \|{}^tB^{{\bf l}}. \n\|^{-\beta} \mu_\gamma(\Delta^{\bf l})\bigg)\\
        & \leq C_1 e^{-k\beta_2} \sum_{\n \in \mathbb{Z}^d} \|f_{\n}\|_2.
    \end{align}
    While the first two inequalities follow from regularity assumptions \eqref{regularity:f,h}, we use the following large deviation estimate to justify the last inequality

    \begin{equation}\label{eq:large-deviation}
    \mu_\gamma\left(\left\{ x \in \Delta_\gamma: \; \frac{1}{k}\log\frac{\|{}^tB_k(x) .\n\|}{\|\n\|}- \chi< -\epsilon \right\}\right) \leq Ce^{-\beta_1 k}
    \end{equation}
    where $\chi$ is the top Lyapunov exponent of $((\hat{\mathcal{R}}^{\gamma})^{-1}, {}^tB)$ defined below. 
    
    It is well known by the work of Dolgopyat \cite{Dolgopyat_2004} that the uniqueness of u-states implies the large 
    deviation principle. We outline the arguments and give the necessary references for the convenience of the reader.
    
    It is easy to verify that the matrices ${}^tB: \Delta_\gamma \to Sp(2g, 
    \mathbb{Z})$ form a cocycle, which is conjugate to a locally constant 
    cocycle (see \cite[Chapter 6]{VianaLyapunovExponentsBook}), over the inverse of the natural extension $\hat{\mathcal{R}}^{\gamma}: \hat{\Delta}_\gamma \to \hat{\Delta}_\gamma$ of $\mathcal{R}^{\gamma}$. The transformation $\hat{\mathcal{R}}^\gamma$, 
    being the natural extension of $\mathcal{R}^{\gamma}: \Delta_\gamma \to \Delta_\gamma$ with its natural invariant measure (induced by $\mu_\gamma$), enjoys the same 
    bounded distortion properties as $\mathcal{R}^{\gamma}$. In 
    \cite{AvilaVianaSimplicity} the 
    authors show the uniqueness of u-states under the assumptions of twisting and pinching of the monoid generated by the 
    image of the cocycle. 
    The monoid generated by the image of ${}^tB$ is conjugate to the monoid generated by the image of $B^{-1}$ (by symplecticity). Therefore, the results of \cite{AvilaVianaSimplicity} imply that the u-state for the cocycle $((\hat{\mathcal{R}}^{\gamma})^{-1}, {}^tB)$ is unique. This in turn puts us in a position to apply the ideas of Dolgopyat \cite{Dolgopyat_2004}, which in this specific case are best illustrated in the proof of \cite[Theorem 25]{AvilaDelecroixweak}. We note that although the 
    notion of u-states is not directly mentioned in \cite{AvilaDelecroixweak}, the proof of the large deviation estimate in \cite{AvilaDelecroixweak} relies on equality (2) in the proof of Theorem 25 of that paper which is an immediate corollary of the uniqueness of u-states. The upshot of the above discussion is that the estimate \eqref{eq:large-deviation} holds.

Positivity of $\chi$ is also guaranteed by the results of \cite{AvilaVianaSimplicity}. Thus
\begin{equation}
    \sum_{{\bf l} \in \Omega^{k}} \|{}^tB^{{\bf l}}. \n\|^{-\beta} \mu(\Delta_\gamma^{\bf l}) \leq Ce^{-\beta_1 k}+ \left\| \n \right\|^{-\beta} e^{-\beta k (\chi-\epsilon)} \leq Ce^{-\beta_2 k}
\end{equation}
for some $\beta_2>0$ which finishes the proof of the last inequality in \eqref{exponential mixing ineq}. 
\end{proof}

\begin{cor}
    The invariant measure $\sum_{\pi \in \mathfrak{R}}\mu_{Z}\big\vert_{\Delta_{\pi}} \times m_{\mathbb{T}^{2g}_{\pi}}$ is ergodic for $\mathcal{R}^Z_{\mathbb{T}^d}$.
\end{cor}

\begin{prop} \label{large-ergodic-components}Let $\mathfrak{R}$ and $\pi$ be as before, then there exists  $m_{0}:=m_{0}(\mathfrak{R})>0$ such that for every prime number $p$ the measure of any ergodic component of $\sum_{\pi \in \mathfrak{R}}\mu_{Z}\big\vert_{\Delta_{\pi}} \times \nu^{(\pi)}_p$ is at least $m_0$. 
\end{prop}
\begin{proof} Let $\pi$ and $\gamma$ be as before. We first show the statement for the measure $\mu_\gamma \times \nu_{p}^{(\pi)}$. The corresponding statement for the original measure $\sum_{\pi \in \mathfrak{R}}\mu_{Z}\big\vert_{\Delta_{\pi}} \times \nu^{(\pi)}_p$ follows as the ergodic components of $\mu_\gamma \times \nu_{p}^{(\pi)}$ are obtained from those of  $\sum_{\pi \in \mathfrak{R}}\mu_{Z}\big\vert_{\Delta_{\pi}} \times \nu^{(\pi)}_p$ via acceleration of $\mathcal{R}^{Z}_{\mathbb{T}^d}$. Let $E$ be a $\mu_\gamma \times \nu_{p}^{(\pi)}$-positive measure invariant subset for the action of $\mathcal{R}^{\gamma}_{\mathbb{T}^d}$. Ergodicity of $(\mathcal{R}^{\gamma}, \mu_{\gamma})$ implies that the projection of $E$ onto $\Delta_\gamma$ has full $\mu_\gamma$ measure. For $x\in \Delta_\gamma$ let $E_x:=\{ y \in Q_p: (x, y) \in E\}$. As there exists only finitely many subsets of the set $Q_p$, there exists a fixed subset $V \subset Q_p$ with the property that $F_V:=\{x \in \Delta_\gamma : E_x =V\}$ has positive $\mu_\gamma$ measure. Let $x_0 \in F_V$ be a Lebesgue-density point. Then
\begin{equation}
    \lim_{n \to \infty}\frac{\mu_\gamma( \Delta_n(x_0) \cap F_V)}{\mu_\gamma(\Delta_n(x_0))}=1,
\end{equation}
where $\Delta_n(x_0)$ is the element of the partition $\{\Delta^{\bf l}: {\bf l} \in \Omega^{n}\}$ of $\Delta_\gamma$ that contains $x_0$. Bounded distortion \eqref{bounded-distortion} then yields 
\begin{equation}
    \lim_{n\to \infty} \mu_\gamma\big((\mathcal{R}^{\gamma})^n(\Delta_n(x_0) \cap F_V)\big)=1.
\end{equation}
As $E$ is invariant, we have  
\begin{equation}
    E \supset(\mathcal{R}^{\gamma}_{\mathbb{T}^d})^{n}\big((\Delta_n(x_0) \cap F_V) \times V\big)= \big((\mathcal{R}^{\gamma})^{n}(\Delta_n(x_0) \cap F_V), B_n(x_0). V\big).
\end{equation}
Therefore, by letting $n$ go to infinity and using the fact that the fibers $E_x$ of $E$ have the same cardinality almost everywhere we get that $E= \Delta_\gamma \times U$ for some $U \subset Q_p$. Such $U$ must then be invariant under the action of the group $O$ generated by matrices $B^{(\ell)}=B_{\gamma \ell \gamma}$. 

Gutierrez showed in \cite{Gutierrez} that $O$ is of some finite index $k_0$ in $Sp(2g, \mathbb{Z})$, the group of matrices preserving the symplectic form on $H(\pi)$ and the lattice $\mathbb{Z}^d \cap H(\pi)$. Pick an arbitrary element $y \in Q_p$. $Sp(2g, \mathbb{Z})$ acts transitively on $Q_p \subset \mathbb{T}^{2g}_{\pi}$. Hence 
\begin{equation}
  [Sp(2g, \mathbb{Z}): Stab(y)]=|Sp(2g, \mathbb{Z}). y|= |Q_p|.
\end{equation}
Similarly,
\begin{equation}
    |O.y|= [O: O \cap Stab(y)] \geq \frac{[Sp(2g,\mathbb{Z}): Stab(y)]}{[Sp(2g, \mathbb{Z}): O]}= \frac{1}{k_0} |Q_p|.
\end{equation}
In particular, this shows that $\mu_\gamma \times \nu_{p}^{(\pi)}(E) \geq \frac{1}{k_0}$. Since this holds for every invariant set $E$, it concludes the proof of the proposition with $m_0=\frac{1}{k_0}$. 
\end{proof}

\subsection{Twisted Birkhoff sums}

In this subsection, we show that the rate of growth of twisted (exponential) Birkhoff sums of a function $f$ is bounded by the rate of growth of the norm of the evolution of $f$ under the twisted cocycle. To this end, we need the following definitions.

\begin{definition}
For a dynamical system $(T,X)$ and measurable functions $\zeta: X \to \mathbb{R},f:X\to \mathbb{C}$, the exponential Birkhoff sum of $f$ with respect to the twist parameter $\zeta$  up to time $n$ is given by
\begin{equation}
    S_{n}(f, \zeta,x):=\sum\limits_{k=0}^{n-1} \exp\big(2\pi i S_{k}(\zeta,x)\big) f(T^{k}(x))
\end{equation}
where $S_k(\zeta,x)=\sum\limits_{j=0}^{k-1}\zeta(T^j(x))$ is the ordinary Birkhoff sum of $\zeta.$
\end{definition}

\begin{definition} \label{col}
    For a matrix with positive entries, we let
        \begin{equation}
            \col(A):= \max _{i,j,k}  \frac{A_{ij}}{A_{kj}}.
        \end{equation}
    \end{definition}
Notice that for a matrix $B$ with nonnegative entries we have 
\begin{equation}
        \col(AB) \leq \col(A).
\end{equation}
We remark that if $h \in \mathbb{R}^{d}_{+}$ then
\begin{equation}
        \max\limits_{i, j} \frac{\langle Ah, e_i\rangle}{\langle Ah, e_j \rangle} \leq \col(A).
    \end{equation}

\begin{prop}\label{exponent-comparison}
For Lebesgue almost every $\lambda \in \Delta$ and every $f \in \mathbb{C}^d$ the following inequality holds
\begin{equation}
     \limsup\limits_{N\to \infty} \sup\limits_{x\in I} \frac{\log \big\vert S_N(f,\zeta,x)\big \vert}{\log N} \leq 
     \max\bigg\{0, \frac{1}{\chi_1} \limsup\limits_{m\to \infty} \frac{\log \|\tw^{Z}_m(\lambda, \pi, \zeta). f\|}{m}\bigg\},
\end{equation}
where $\chi_1$ denotes the top Lyapunov exponent of the Rauzy-Veech-Zorich cocycle. 
    
\end{prop}

\begin{proof}

Let $(\lambda, \pi)$ be infinitely renormalizable (under Zorich renormalization). Let $n_1<n_2< \cdots$ be a sequence of positive integers with the following properties 
\begin{itemize}
    \item[i)] {\it (Subexponential growth of partial quotients) for every $\epsilon>0$ there exists $C_\epsilon>0$ such that}
    \begin{equation}
    \|B^{Z}_{n_{j+1}-n_j}\big(\mathcal{R}_{n_{j}}^{Z}(\lambda, \pi)\big)\|_{1} \leq C_\epsilon e^{n_j\epsilon},
    \end{equation}
    \item[ii)] {\it ($L$-balanced property) for $h=(1,1,\ldots, 1)^{t} \in \mathbb{R}^d, h^{(k)}:= B^{Z}_{n_k}(\lambda, \pi).h$ we require that}
    \begin{equation} 
        \min_{\alpha \in \mathcal{A}} |h^{(k)}_{\alpha}| \geq \frac{1}{L} \max_{\alpha \in \mathcal{A}} |h^{(k)}_{\alpha}|.
    \end{equation}
\end{itemize}

We first show that the above conditions can be guaranteed for a full measures set of $(\lambda, \pi)$. Let $\mathfrak{R}$ be the Rauzy diagram of $\pi$ as before and take $\gamma \in \Pi(\mathfrak{R})$ be a loop starting and ending at $\pi$ such that all the entries of $B_\gamma$ are positive. Furthermore, we assume that the starting and ending 
arrows of $\gamma$ are different. Let $\Delta_{\gamma.\gamma}$ be as in Section \ref{sec: Twisted Cocycle} and $(\lambda, \pi)$ be such that its renormalization trajectory visits $\Delta_{\gamma.\gamma}$ infinitely many times. Let $m_i$ be the $i-th$ visiting time of $\Delta_{\gamma.\gamma}$ by $(\lambda, \pi)$ under $\mathcal{R}_{Z}$. Then we take $n_i= m_i + |\gamma|_{Z}$, where $|\gamma|_Z$ denotes the 
number of Zorich operations needed to represent $\gamma$. That is, $n_i$ is the first time the trajectory of $(\lambda, \pi)$ visits $\Delta_\gamma$ after having visited $\Delta_{\gamma. \gamma}$ for the $i$-th time. By Definition \ref{col} and the fact that 
\begin{equation}
    B^{Z}_{n_i}(\lambda, \pi)= B_\gamma B^{Z}_{m_i}(\lambda, \pi),
\end{equation}
we can ensure, by taking $L= col(B_\gamma)$, that condition $ii)$ is satisfied for this choice of $\{n_i\}_{i=1}^{\infty}$. In order to guarantee condition $i)$, we demand that $(\lambda, \pi)$ be such that $\mathcal{R}_{Z}^{n_1}(\lambda, \pi) \in \Delta_{\gamma.\gamma}$ is Oseledets regular (see Subsection \ref{basic-ergodic-theory}) for the acceleration of the Zorich cocycle defined over $\mathcal{R^{\gamma.\gamma}}$. Then we see that
\begin{equation}
    \lim_{i \to \infty} \frac{\log \|B^{Z}_{m_{i+1}-m_i}(\lambda, \pi)\|}{i}=0
\end{equation}
which easily implies condition $i)$.

We now turn to the proof of the proposition under the aforementioned full measure conditions. Let $(\lambda^{(k)}, \pi^{(k)}):= \mathcal{Q}^{n_k}_Z(\lambda, \pi)$. We denote by $T^{(k)}$ the IET corresponding to 
$(\lambda^{(k)}, \pi^{(k)})$ and let $I^{(k)}:=[0, |\lambda^{(k)}|)$ be its domain of 
definition. Recall that $I^{(k)}$ admits a natural partition to subintervals $I^{(k)}_{\alpha}, \alpha \in \mathcal{A}$. We call the set
\begin{equation}
\bigcup\limits_{j=0}^{h^{(k)}_\alpha -1}T^{j}(I^{(k)}_\alpha),
\end{equation}
the $k$-tower corresponding to the letter $\alpha \in \mathcal{A}$. A $k$-tower orbit segment is then a trajectory of the form $\{y, T(y), \ldots, T^{h^{(k)}_\alpha -1}(y) \}$, where $y\in I^{(k)}_{\alpha}$ for some $\alpha$. We now aim to decompose every trajectory $\{x, T(x), \ldots, T^{N-1}(x)\}$ into a union of special orbit segments, each consisting fully of unions of $k$-tower orbit segments for some $k$. To this end, we use the construction of \cite[\S 2.2.3]{MMYRothtype}. While we do not reproduce their arguments here, we recall a summary of their exposition. Given $x \in I, N\in \mathbb{N}$,
the authors find $k\geq 0$ and a sequence 
\begin{equation}
    0=r^{-}_0 \leq r^{-}_1 \leq \cdots \leq r^{-}_{k+1}=r^{+}_{k+1}\leq r^{+}_{k} \leq \cdots
    \leq r^{+}_0=N
\end{equation}    
in such a way that $\{T^{r^{+}_{j+1}}(x), \ldots, T^{r^{+}_{j}-1}(x)\}$ and $\{T^{r^{-}_{j}}(x), \ldots, T^{r^{-}_{j+1}-1}(x)\}$ are each unions of at most $\|B^{Z}_{n_{j+1}-n_j}\big(\mathcal{R}_{n_{j}}^{Z}(\lambda, \pi)\big)\|_{1}$ many $j$-tower orbit segments where for a matrix $A$ we let $\|A\|_1:= \sum \limits_{i,j} |A_{ij}|$. We can now express the twisted Birkhoff sum as
    \begin{align}
         S_N(f, \zeta, x) &= \sum \limits_{j=0}^{k} \exp\big(2\pi i S_{r^{-}_j}(\zeta, x)\big) S_{r^{-}_{j+1}-r^{-}_{j}-1}\big(f, \zeta, T^{r^{-}_{j}(x)}\big)\\
        &+ \sum \limits_{j=0}^{k} \exp\big(2\pi i S_{r^{+}_{j+1}}(\zeta, x)\big) S_{r^{+}_{j}-r^{+}_{j+1}-1}\big(f, \zeta, T^{r^{+}_{j+1}(x)}\big).
    \end{align}
As the twisted Birkhoff sum over a $j$-tower is, by definition, bounded by $\|\mathcal{B}^{Z}_{n_{j}}(\lambda, \pi, \zeta). f\|$ we get
\begin{equation}
        |S_N(f, \zeta, x)| \leq \sum \limits_{j=0}^{k} 2\|B^{Z}_{n_{j+1}-n_j}\big(\mathcal{R}_{n_{j}}^{Z}(\lambda, \pi)\big)\|_{1}\cdot\|\mathcal{B}^{Z}_{n_{j}}(\lambda, \pi, \zeta). f\|.
\end{equation}
If we denote 
\begin{equation}
    \beta:=\limsup\limits_{m\to \infty} \frac{\log \|\tw^{Z}_m(\lambda, \pi, \zeta). f\|}{m},
\end{equation}
then for every $\epsilon>0$ there exists $C'_\epsilon>0$ such that  
\begin{equation}
    \|\mathcal{B}^{Z}_{n_{j}}(\lambda, \pi, \zeta). f\| \leq C'_\epsilon e^{(\beta+\epsilon)n_j}.
\end{equation}
The above inequalities yield
\begin{equation}
        |S_N(f, \zeta, x)| \leq 2\sum \limits_{j=0}^{k} C^{}_{\epsilon}C_{\epsilon}^{\prime} e^{n_{j}(\beta+2\epsilon)},
\end{equation}
which implies that if $\beta<0$, then
\begin{equation}
    |S_N(f, \zeta, x)| \leq kM,
\end{equation}
for some constant $M$ independent of $N$. On the other hand, for $\beta>0$ we get
\begin{equation}
    |S_N(f, \zeta, x)| \leq 2kM_\epsilon e^{n_k(\beta+2\epsilon)},
\end{equation}
for some $M_\epsilon>0$. Since by the definition of $k$ the orbit $\{x, T(x), \ldots, T^{N-1}(x) \}$ contains at least one $k$-tower and the lengths of $k$-towers are $L$-balanced 
\begin{equation}
    N \geq \min_{\alpha \in \mathcal{A}} \vert h^{(k)}_{\alpha}\vert \geq \frac{1}{L} \|h^{(k)}\| \geq   \frac{1}{L} C'_{\lambda, \epsilon} e^{n_k(\chi_1-\epsilon)}.
\end{equation}
Letting $\epsilon$ go to zero implies the desired result. 

\end{proof}

\section{A suspension construction}\label{sec: IETs with positive exponent}

In this section, we introduce a foliation $\f_p$ by line segments of a subspace of $\Delta$ corresponding to each prime $p$. We then give a method to construct a translation surface $X$ out of $\n\in \mathbb{Z}^d$ and a
leaf of the foliation $\f_p$. It allows us to relate the exponential sums with twist factors $\n \in \mathbb{Z}^d$ to the ordinary Birkhoff sums of a one parameter family of IETs (see Definition \ref{Stopping time}, Lemma \ref{(p,n)-iterate}) which correspond to Poincaré return maps to a fixed transversal section of the one parameter family of directional translation flows on $X$.

Let $\pi=(\pi_t, \pi_b)$ an irreducible permutation such that $\pi_b \circ \pi_t^{-1}(d)=1$. Hereinafter, we let $p$ be a prime and $\n \in \mathbb{Z}^d$ such that $p\nmid n_{\alpha_t}$.  

We consider IETs corresponding to $\pi$ and the following length parameters 

\begin{equation}
    \Delta_{p}:= \big\{\lambda \in \Delta | \lambda_{\alpha_t}> \frac{p-1}{p}\big\}.
\end{equation}
Let $\mathcal{A}_{t}:=\mathcal{A}\setminus \{\alpha_t\}$. Then we define the mapping $\f_p:\Delta^{\mathcal{A}_t} \times [0,1) \to \Delta_p$ by
\begin{equation}
    \f_p(\hat{\lambda}, s):= \big(\frac{s}{p}\hat{\lambda}, (1-s)+\frac{s(p-1)}{p}\big).
\end{equation}

$\f_p$ induces a foliation on $\Delta_{p}$ by a family of line segments.

\begin{definition}\label{Stopping time} For $\lambda \in \Delta_p$ we define the $(p,\n)$-stopping time of a point $x \in I$ under $T_{\lambda, \pi}$ to be the least positive integer $\tpn(x)$ (whose existence is guaranteed by Lemma \ref{(p,n)-iterate}) such that 
\begin{equation} \label{tpn eq}
    \sum\limits_{j=0}^{\tpn(x)-1} \n(T_{\lambda, \pi}^{j}(x)) \equiv 0 \pmod{p}
\end{equation}
    
\end{definition}

\begin{lemma} \label{(p,n)-iterate}
    Let $\lambda \in \Delta_p$ and $\n \in \mathbb{Z}^d$ such that $p \nmid n_{\alpha_{t}}$. Then the following hold
    \begin{itemize}
        \item[(i)] $\tpn(x)$ is well defined and everywhere bounded above by $2p-1$,
        \item[(ii)] the map $\spn(\lambda, \pi): x\mapsto T_{\lambda,\pi}^{t}(x)$, where $t=\tpn(x)$, is a $p(d-1)+1$-IET,
        \item [(iii)] for every $ \alpha \in \mathcal{A}_t$, there exists a unique integer $0\leq \ell_\alpha\leq p-1$ so that for every $x\in \bigcup\limits_{j=0}^{p-1}T^{-j}(I_\alpha)$, $\tpn(x)\overset{p}{\equiv} \ell_\alpha$.
    \end{itemize} 
\end{lemma}
\begin{proof}
Let  $I_\alpha$ be as in \ref{sec: Prelim}. Since $\lambda_{\alpha_t}> \frac{p}{p+1}$ and that $\pi_b \circ \pi_t^{-1}(d)=1$, it is clear that for every nonzero $-p\leq j \leq p$ $T^{j}(I_\alpha) \subset I_{\alpha_t}$. 

For $x\in I_\alpha$, the Birkhoff sums of $\n$ (viewed as a locally constant function) along the orbit of $x$ up to the $p$-th iterate comprise the following numbers
\begin{equation}
    n_\alpha, n_\alpha+n_{\alpha_t}, \ldots, n_\alpha+ (p-1)n_{\alpha_t}.
\end{equation}
One and only one of the above is divisible by $p$ (as $p \nmid n_{\alpha_t}$), which we denote by $n_\alpha + m_\alpha n_{\alpha_t}$.

Therefore, for $0\leq j \leq p-1$ and $x\in T^{-j}(I_\alpha) \subset I_{\alpha_t}$ we can let
\begin{equation}
\tpn(x):=
\begin{cases}
    m_\alpha+1  &\textit{for} \;\;\;\; 0\leq j\leq m_\alpha,\\
    m_\alpha+p+1  & \textit{otherwise.} 
\end{cases}    
\end{equation}
For $x \in \Tilde{I}:=[0,1) \setminus \bigcup\limits_{j=0}^{p-1}T^{-j}([0,1)\setminus I_{\alpha_t})$, 
the iterates until time $p-1$ of $x$ all remain inside the largest interval $I_{\alpha_t}$. Hence $\tpn(x)=p$ satisfies Definition \ref{tpn eq} in this case. In view of the fact that $\tpn$ is constant on each of the intervals $\Tilde{I}$ and $T^{-j}(I_\alpha)$ for $\alpha \in \mathcal{A} \setminus \{\alpha_t\}$ and $0\leq j \leq p-1$, it may be readily 
verified that $\spn$ is a translation on each of these 
intervals. To prove injectivity we remark that for any $\alpha \in \mathcal{A} \setminus \{\alpha_t\}$ the set
\begin{equation}
    J_\alpha:= \bigcup \limits_{j=0}^{p-1}T^{-j}(I_\alpha)
\end{equation}
will be mapped under $\spn$ to the set
\begin{equation}
    \Tilde{J}_\alpha:= \bigcup \limits_{j=1}^{p} T^{j}(I_\alpha)
\end{equation}
which is a union of $p$ disjoint intervals. The interval $\Tilde{I}$ on the other hand, gets mapped (via a translation) to $[0, |\Tilde{I}|)$. Therefore, $\spn$ is an interval exchange transformation and its number of intervals is $p(d-1)+1$.
\end{proof}

Let $\lambda,\pi,p,\n$ be as before and $\zeta_k:=\frac{k}{p}\n$. We may drop the dependence on $\lambda, \pi$ whenever it arises no confusion. Given a locally constant function $f$ with respect to  $T$ and  $0\leq k\leq p-1$, 
\begin{equation}
     \Bir_k(f)(x):= S_{\tpn(x)}(f, \zeta_k, x)= \sum\limits_{j=0}^{\tpn(x)-1} \exp\big(2\pi i S_{j}(\zeta_k,x)\big) f(T^{j}(x))
\end{equation}
defines a locally constant function with respect to $\spn$. It is easy to verify that $\Bir_k:\mathbb{C}^d\to \mathbb{C}^{p(d-1)+1}$ is a linear transformation. We emphasize that the exponential Birkhoff sums in the definition of $\Bir_k$ are with respect to the twist parameter $\zeta_k$.

\begin{lemma}(Untwisting lemma)\label{Small kernel surjection}
For each $1\leq k \leq p-1$ the kernel of the linear map $\Bir_k$ is one dimensional and the linear transformation 
    
    \begin{equation}
      \bigoplus\limits_{k=0}^{p-1}\Bir_k: \bigoplus\limits_{k=0}^{p-1}\mathbb{C}^{d} \to \mathbb{C}^{p(d-1)+1}
    \end{equation}
is surjective. 
    
\end{lemma}

\begin{proof} 
Note that $\tpn(x)$ is, by definition, designed so that the total amount of twist along the orbit segment $\{ x, T(x), \ldots, T^{\tpn(x)-1}(x)\}$ is an integer multiple of $2\pi$. Let $\ell$ be a positive integer and for each $0\leq j\leq \ell$ set 

\begin{equation}
    m_j:= \sum_{i=0}^{j-1} \tpn(\spn^{i}(x)), \quad x_j:=\spn^{j}(x)=T^{m_j}(x)
\end{equation}
We now decompose the twisted Birkhoff sum of the function $f$ along the orbit $\{x, T(x), \ldots, T^{m_\ell-1}(x)\}$ into a sum of twisted Birkhoff sums over $\ell$ orbit subsegments $\{T^{m_{j-1}}(x), \ldots, T^{m_{j}-1}(x)\}$, $1 \leq j \leq \ell\}$. 
Clearly, for $m_j\leq n<m_{j+1}$, 
\begin{equation}
    S_n(\zeta_k,x)=S_{n-m_j}(\zeta_k,x_j)+S_{m_j}(\zeta_k,x),
\end{equation}
and since $S_{m_j}(\zeta_k,x)\in \mathbb{Z}$, $\exp(2\pi i S_n(\zeta_k,x))=\exp(2\pi i S_{n-m_j}(\zeta_k,x_j))$. Therefore, 
\begin{align}
    S_{m_\ell}(f, \zeta_k,x)& =\sum\limits_{n=0}^{m_\ell-1} \exp\big(2\pi i S_{n}(\zeta_k,x)\big) f(T^{n}(x))\\
    &= \sum\limits_{r=0}^{\ell-1}\sum\limits_{n=m_r}^{m_{r+1}-1} \exp(2\pi i S_{n-m_r}(\zeta_k,x_r))f(T^{n-m_r}(x_r))\\
    &= \sum\limits_{r=0}^{\ell-1}\Bir_k(f)(\spn^r(x)).
\end{align}

The upshot is that exponential Birkhoff sums of the function $f$ under $T$ with respect to twist parameters $\zeta_k$ at certain special times are equal to the ordinary Birkhoff sums of $\Bir_k(f)$ under $\spn$.

For $\alpha \in \mathcal{A} \setminus \{\alpha_t\}$, we have that $\Bir_k(\one_\alpha)$ is the function that attains the value $\exp\big(2\pi i j \frac{kn_{\alpha_t}}{p}\big)$ on $T^{-j}(I_\alpha)$ for $0 \leq j \leq p-1$ and the value $0$ on the complement of $\bigcup\limits_{j=0}^{p-1} T^{-j}(I_\alpha)$. So by Vandermonde's identity the linear span of the space of functions $\Bir_{k}(\one_\alpha), \; k \in \{0,1, \ldots, p-1\}$ is the space of all locally constant functions (for $\spn$) that are supported on $J_{\alpha}$.

For $\alpha_t$ on the other hand, 
\begin{align}
    \sum \limits_{k=0}^{p-1}\Bir_k(\one_{\alpha_t})(x)&=  \sum \limits_{k=0}^{p-1}\sum\limits_{j=0}^{\tpn(x)-1} \exp\big(2\pi i S_{j}(\zeta_k,x)\big) \one_{\alpha_t}(T^{j}(x))\\
    &= \sum \limits_{j=0}^{\tpn(x)-1}\sum\limits_{k=0}^{p-1} \exp\big(2\pi i S_{j}(\zeta_k,x)\big) \one_{\alpha_t}(T^{j}(x))\\
    &= p \one_{\alpha_t}(x)
\end{align}
Since $I_{\alpha_t}=\bigcup\limits_{j=1}^{p} T^{-j}([0,1)\setminus I_{\alpha_t}) \bigcup \Tilde{I}$, we have   $\one_{\Tilde{I}}=\one_{\alpha_t}-\sum\limits_{\alpha \neq \alpha_t}\sum \limits_{j=1}^{p-1} \one_{T^{-j}(I_{\alpha})}$. Hence, we can express every locally constant function with respect to $\spn$ as a linear combination of functions of the form
$\Bir_k(\one_\alpha), \; \alpha \in \mathcal{A}, \; k \in \{0,1,\ldots, p-1\}$.  
\end{proof}

\subsection{Geometric realization}\label{translation-surface}
Our aim in this subsection is to construct the aforementioned translation surface $X=X(\n, p, \hat{\lambda})$ and realize the one parameter family of IETs $\spn(\f_p(\hat{\lambda}, s), \pi)$ as the Poincaré return maps of the directional translation flows to a certain fixed transversal on $X$.

Let  $\lambda^{*}:=\f_p(\hat{\lambda},\frac{p+1}{2p})$, $\Tilde{\lambda}:= (\frac{1}{2} \hat{\lambda}, \frac{1}{2})$. Then the length of the last interval of the IET $\spn(\lambda^{*}, \pi)$ equals $1/2$. Let $P$ be the square $[0,\sqrt{2}/2] \times [0,\sqrt{2}/2]$. We identify the boundary points of $P$ as follows 
\begin{itemize}
    \item Top and bottom sides: $\big(x,\sqrt{2}/2\big)\sim\big(\sqrt{2}\spn(\frac{x}{\sqrt{2}})-\sqrt{2},0\big)$,
    \item Left and right sides: $\big(0,y\big)\sim \big(\sqrt{2}/2,y\big)$.
\end{itemize}
We denote the translation surface obtained from $P$ via the above side identifications by $X(\n, p, \hat{\lambda})$ (or by $X$ if there is no ambiguity). Indeed, one can view the 
square as a $(2p(d-1)+2)$-gon which has a pair of $p(d-1)$ degenerate sides lined up along the 
top and bottom edges of $P$ (see Figure \ref{fig:surface}). 

By replacing $\spn$ with $T_{\Tilde{\lambda}, \pi}$ in the definition of $X$, we obtain another translation surface which we denote by $Y \label{definition of Y}$. Then the Riemann surface $X$ is a $p$-fold cover of the surface $Y$ ({see item (iii) of Lemma \ref{(p,n)-iterate} and Figure \ref{fig:surface})}. We denote the covering map by $\Pi:X\to Y$.

\begin{figure}[t]
    \centering
    \includegraphics[width=.6\textwidth]{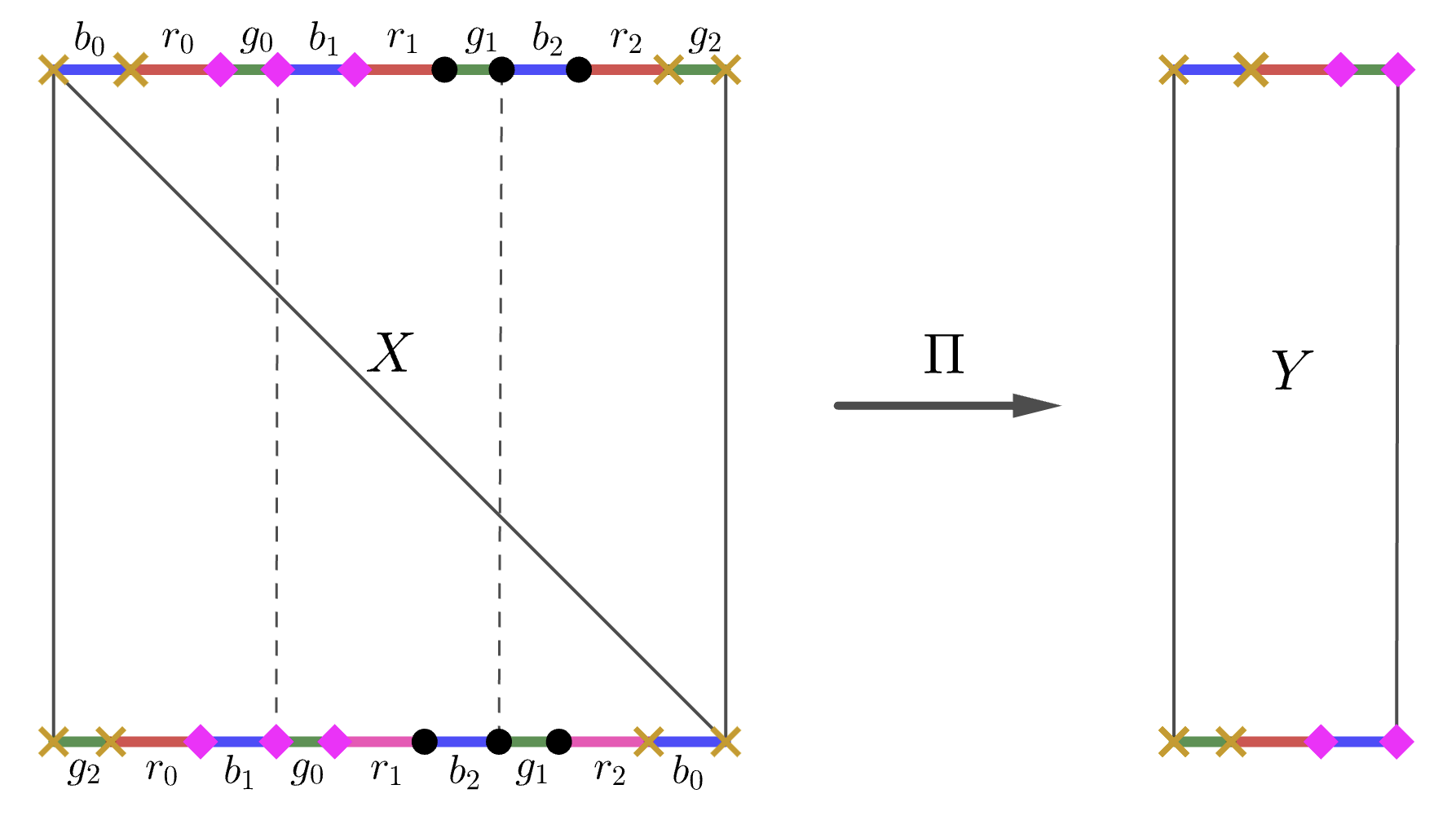}
    \caption{Example of surfaces $X,Y$ for $p=3,d=4$ and $\n=(3,1,2,1)$}
    \label{fig:surface}
\end{figure}

\begin{lemma}(genus bound)\label{genus bound} The genus of the surface $X$ satisfies the following inequalities
\begin{equation}
     p(g-1)+1\leq g(X)\leq p\frac{d}{2}.
\end{equation}
\end{lemma}

\begin{proof} The upper bound follows from \eqref{genus and singularities}. The lower bound is a consequence of the Riemann-Hurwitz formula applied to the covering map $\Pi:X \to Y$.
\end{proof}
\begin{lemma} \label{Order of singularities} Every conical singularity of order $m$ of the surface $Y$ gives rise to:
\begin{itemize}
    \item[i)] either a singularity of order $p(m+1)-1$,
    \item[ii)] or a family of $p$ singularities each of which of order $m$,  
\end{itemize}
for $X$. 
\end{lemma}
\begin{proof}
We define the mapping $D_p:X\to X$ as follows
\begin{equation}
    D_p (x, y) = \big(x, y+\frac{\sqrt{2}}{2p}\big),
\end{equation}
where we view the addition operation on the second coordinate  modulo $\frac{\sqrt{2}}{2}$. $D_p$ is a deck transformation for the covering map $\Pi : X \to Y$, and therefore, acts on the set $\Sigma_X$ of singularities of $X$ via permutation. $D_p^p=id$ implies that each orbit of the action of $D_p$ on $\Sigma_X$ consists of either $1$ or $p$ elements. As $D_p$ preserves the abelian differential induced by $dz$ on $X$, the orders of singularities of the elements in a cycle are equal. In case a cycle consists only of one element, the corresponding cone angle is $p$ multiplied by the cone angle of the corresponding singularity in $Y$, i.e, $2\pi(p(m+1))$. Otherwise, the cycle consists of exactly $p$ elements whose corresponding cone angle is equal to $2\pi(m+1)$. It concludes the proof. 
\end{proof}

Let us choose $s(\theta)$ so that the return map of the translation flow in direction $\theta$ (measured as the angle between the given direction and the $y$-axis) on $X$ to the diagonal (connecting the top left corner of $P$ to its lower right corner) is $\spn(\f_p(\hat{\lambda}, s(\theta)), \pi)$. 

\begin{lemma}\label{absolute continuity}
The mapping $\theta \mapsto s(\theta)$ is absolutely continuous. 
    
\end{lemma}
\begin{proof} It is clear, by the definition of $\f_p$, that
\begin{equation}
    \lambda_{\alpha_t}=1-\frac{s}{p},\quad 
    |\Tilde{I}|=1-p(1-\lambda_{\alpha_t})=1-s.
\end{equation}
By an elementary computation (using the law of sines) we get
\begin{equation}
    |\Tilde{I}|=\frac{\sin(\theta)}{\cos(\theta)+\sin(\theta)},
\end{equation}
and therefore 
\begin{equation}
    s(\theta) = \frac{1}{1+\tan(\theta)},
\end{equation}
which is a $C^{1}$ function over the interval $(0,\frac{\pi}{2})$. Absolute continuity immediately follows. 
\end{proof}

\section{Proof of Main Theorem}\label{sec:proof-main-theorm}

In this section, we apply theorems of \cite{Eskin-Kontsevich-Zorich} and \cite{CE13} to get lower bounds on the growth of Birkhoff sums for IETs corresponding to a leaf of $\f_p$. Then, we deduce exponential growth for the twisted cocycle for almost every parameter on that leaf. Let \[\mathcal{M}=\mathcal{M}(\n, p, \hat{\lambda}):= \overline{SL(2, \mathbb{R}). X(\n, p, 
\hat{\lambda})}.\] 

Henceforth, we fix $\n, p, \lambda$ and drop the subscripts involving 
these parameters. By the seminal work of Eskin, Mirazkhani, and Mohammadi \cite[Theorem 2.1]{Eskin-Mirzakhani-Mohammadi}, $\mathcal{M}$ is an affine invariant submanifold of its ambient moduli space.
 We use an inequality for the sum of Lyapunov exponents of $\mathcal{M}$ derived from \cite[Theorem 1]{Eskin-Kontsevich-Zorich}. The formula given in \cite[Theorem 1]{Eskin-Kontsevich-Zorich} was proved under the assumption of regularity of the $SL(2, \mathbb{R})$-invariant suborbifolds. It was later proved in \cite{Avila_Matheus_Yoccoz} that this assumption holds for every invariant suborbifold (for a more recent stronger version of the result of \cite{Avila_Matheus_Yoccoz}, see \cite{Ben_Dozier}).

\begin{thm} \label{Sum of exponents}(\cite[Theorem 1]{Eskin-Kontsevich-Zorich})
Let $\mathcal{M}$ be any closed connected $SL(2, \mathbb{R})$-invariant suborbifold of some stratum $\mathcal{H}_1( m_1, m_2, \ldots, m_\kappa)$ of Abelian differentials, where $m_1+\cdots+m_\kappa=2g-2$. The top $g$ Lyapunov exponents of the Hodge bundle $H^{1}_{\mathbb{R}}$ over $\mathcal{M}$ along the Teichmüller flow satisfy the following inequality
\begin{equation}
    \theta_1(\mathcal{M})+\cdots+ \theta_g(\mathcal{M}) \geq \frac{1}{12}\sum \limits_{i=1}^{\kappa} \frac{m_i(m_i+2)}{m_i+1}.
\end{equation}
The leading Lyapunov exponent $\theta_1$ is equal to one.
\end{thm}

\begin{prop}{\label{periodic exponent approximation}}
    Let $\pi$ be an irreducible permutation of genus greater than $1$ and $\mu_{Z}$ the corresponding absolutely continuous invariant measure under Zorich renormalization. Then, there exist $c_0, c_1>0$ such that for every large enough prime $p$ and $\mu_{Z}$ almost every $(\lambda, \pi)$ there exists $E_p(\lambda, \pi) \subset Q_p$ with the following properties:
 \begin{itemize}
        \item[i)] (positive exponent) for every $y \in E_p(\lambda, \pi)$, $\tilde{\chi}^{+}(\lambda, \pi, y)$  exists and is larger than or equal to $c_1$,
        \item[ii)] (positive proportion) $\nu_p(E_p(\lambda, \pi))>c_0$,
 \end{itemize}
where $\tilde{\chi}^{+}$ denotes the pointwise top exponent of the twisted cocycle.
\end{prop}
\begin{rmk}\label{periodic exponent approximation for the ergodic measure}
    The same proof shows that the properties above hold true if we replace $\nu_p$ by $\nu_p^{(\pi)}$. 
\end{rmk}

\begin{proof} Note that Oseledet's theorem applied to the measure $\mu_{Z} \times \nu_p$ implies the existence of $\tilde{\chi}^{+}(\lambda, \pi, y)$ for $\mu_Z$ almost every $\lambda$ and every $y \in Q_p$. We assume that $(\lambda, \pi)$ satisfies the full measure conditions of the proof of Proposition \ref{exponent-comparison}, and that for every $y\in Q_p$, $(\lambda, \pi, y)$ is Oseledet regular for the twisted cocycle and the measure $\mu_Z \times \nu_p$. Moreover, we assume $\lambda$ is such that for every $\n \in Q_p$,  $\spn(\lambda, \pi)$ is the return map to the diagonal of the translation flow in an Oseledets regular (see Subsection \ref{basic-ergodic-theory}) direction for the surface $X(\n, p, \hat{\lambda})$ (this is guaranteed by the result of \cite{CE13} regarding Oseledet's regularity of almost every direction in the translation surface $X$ under the Teichmüller geodesic flow on $\mathcal{M}$, and Lemma \ref{absolute continuity}). We now fix $\n \in \mathbb{Z}^d$ and let $X=X(\n, p, \hat{\lambda})$ as before and denote by $g^*$ the genus of $X$. We assume furthermore that $X \in \mathcal{H}_1(m^{*}_1, \ldots, m^{*}_{\kappa^{*}}), Y \in \mathcal{H}_1(m_1, \ldots, m_\kappa)$ (see \ref{translation-surface} for the definition of $Y$). Lemma \ref{Order of singularities} gives the following lower bound
\begin{equation}\label{eq:lower bound 1-12}
    \frac{1}{12} \sum_{j=1}^{\kappa^{*}} \frac{m^{*}_j(m^{*}_j+2)}{m^{*}_j+1} \geq \frac{p}{12} \sum_{i=1}^{\kappa} \frac{m_i(m_i+2)}{m_i+1} \geq \frac{p}{12} \sum_{m_i\geq 1} (m_i+\frac{1}{2}) \geq \frac{p(4g-3)}{24}.
\end{equation}
Theorem \ref{Sum of exponents} and Lemma \ref{genus bound} imply that
\begin{equation}
     \frac{1}{g^{*}}(\theta_1(\mathcal{M})+\cdots+ \theta_{g^{*}}(\mathcal{M})) \geq \frac{g}{12d}=:2c.
\end{equation}
In view of the fact that the largest exponent $\theta_1=1$, we get 
\begin{equation} \label{bound on the portion of exponents}
    \frac{\;\big\vert\{1 \leq j \leq g^* : \theta_j(\mathcal{M}) \geq c\}\big\vert\;}{g*} \geq c.
\end{equation}
Thus, if we let $V_c \,{\mbox{\large$\subset$}} \, \mathbb{C}^{p(d-1)+1}$ be the Oseledet's subspace corresponding to the exponents that are smaller than $c$, then by \eqref{bound on the portion of exponents}, $\dim(V_c) \leq pd-cg^{*}.$ Lemma \ref{Small kernel surjection} implies that the images of $\Bir_k$ for $1\leq k \leq p-1$ form a direct sum of vector subspaces of dimension $d-1$. Therefore, 

\begin{equation}
\frac{1}{p-1}\bigg\vert\Big\{1\leq k \leq p-1: \Bir_k(\mathbb{C}^d){\mbox{\Large$\nsubset$} } \, V_c\Big\}\bigg\vert \geq \frac{1}{p-1}\bigg(1-\frac{\dim (V_c)}{d-1}\bigg)\geq \frac{c(g-1)}{d}=:c^{*},
\end{equation}
where in the last inequality we assume $p$ is large enough and use the lower bound for $g^*$ from Lemma \ref{genus bound}. Let $k$ be such that $\Bir_k(\mathbb{C}^d) \,{\mbox{\large$\nsubset$}} \, V_c$ and $f \in \mathbb{C}^d$ so that $\Bir_k(f) \notin V_c$. Then, the Birkhoff sums of $\Bir_k(f)$ under $\spn(\lambda, \pi)$ satisfy
\begin{equation}
    \limsup_{n \to \infty} \sup_{x\in I} \frac{\log \vert S_n(\Bir_k(f),x)\vert}{\log n} \geq c.
\end{equation}
Note that the sequence $\big\{S_n\big(\Bir_k(f), x\big)\big\}_{n \in \mathbb {N}}$ is by (the proof of) Lemma \ref{Small kernel surjection}, a subsequence of $S_m(f, \zeta_k, x)$. Now, we are in a position to apply Proposition \ref{exponent-comparison} that implies 
\begin{equation}
    \limsup_{m \to \infty} \frac{\log \|\tw^{Z}_m(\lambda, \pi, \zeta_k). f\|}{m} \geq c \chi_1:=c_1.
\end{equation}
It implies that  
\begin{equation}
    \frac{|\{1 \leq k \leq p-1: \tilde{\chi}^{+}(\lambda, \pi, \zeta_k) \geq c_1\}|}{p-1} \geq c^{*}.
\end{equation}
Repeating this argument for $\n \in (\mathbb{Z}/p\mathbb{Z})^{d}$ with $p \nmid n_d$, we get that the set 
\begin{equation}
    E_p(\lambda, \pi):=\big\{y \in Q_p: \tilde{\chi}^{+}(\lambda, \pi, y) \geq c_1 \big\}
\end{equation}
has cardinality at least $c^{*}(p-1)p^{d-1}$ which implies that for some positive $c_0$, $\nu_p(E_p(\lambda, \pi))>c_0$.

As the union of images of $\Delta_p$ under successive iterations of the Zorich transformation 
has full measure in the base $\Delta \times \mathfrak{R}$, we can immediately upgrade the established property about the top exponent to $\mu_Z$-almost every $(\lambda, \pi) \in \Delta \times \mathfrak{R}$. 
\end{proof}
The following is straightforward from Proposition \ref{periodic exponent approximation}.
\begin{cor} 
There exists a constant $c_2>0$ independent of $p$,  such that for every large enough prime $p$,  
\begin{equation}{\label{exponent lower bound for rational points}}
    \tilde{\chi}^{+}_{\mu_Z \times \nu_p} > c_2.
\end{equation}
\end{cor}
\begin{proof}[Proof of Main Theorem]    
$\mu_Z \times m_{\mathbb{T}^d}$ is the limit of the measures of the form $\mu_Z \times \nu_p$ (that are invariant under the toral Rauzy-Veech-Zorich transformation) in the weak-$^{*}$ topology on measures. Therefore, by \eqref{exponent lower bound for rational points} and using upper semicontinuity of the top Lyapunov exponent, we get
\begin{equation}
    \tilde{\chi}^{+}_{\mu_Z \times m_{\mathbb{T}^d}} \geq \limsup\limits_{p \to \infty} \tilde{\chi}^{+}_{\mu_Z \times \nu_p} \geq c_2>0.
\end{equation}
The corresponding statement for the ergodic measure $\sum_{\pi \in \mathfrak{R}}\mu_{Z}\big\vert_{\Delta_{\pi}} \times m_{\mathbb{T}^{2g}_{\pi}}$ follows in a similar fashion using Remark \ref{periodic exponent approximation for the ergodic measure}.

\end{proof}
\begin{rmk}
    The same proof works, almost verbatim, if we replace $\mu_Z$ with any other measure $\mu$ induced by an ergodic $SL(2, \mathbb{R})$-invariant measure on the moduli space. In fact, in our proof, we only need that $\mu$ to admit absolutely continuous disintegration with respect to the leaves of the foliation $\f_p$, which correspond to unstable horocycle leaves on the moduli space.  By the celebrated work of Eskin-Mirzakhani \cite[Theorem 1.4]{Eskin-Mirzakhani}, all $SL(2, \mathbb{R})$ ergodic invariant measures are affine, and therefore, admit disintegrations to the foliation by unstable horocycle leaves that are absolutely continuous with respect to the Lebesgue measure, which implies absolute continuity of $\mu$ with respect to the leaves of $\f_p$.  
\end{rmk}

\section{Applications}\label{sec: Applications}
In this section, we provide the applications listed in the introduction. 
\subsection{Kontsevich-Zorich conjecture: an observation}
We first mention a conjecture of Kontsevich and Zorich regarding the limit of the second exponent of the Kontsevich-Zorich cocycle over the moduli space of translation surfaces as the genus goes to infinity. We then construct a family of $SL(2, \mathbb{R})$-invariant submanifolds of genus going to infinity for which the limit of the second exponent is bounded by the ratio of the top exponent of the twisted cocycle to the top exponent of the Rauzy-Veech-Zorich cocycle for a certain permutation on $4$ symbols. In addition, we show that the top exponent of the twisted cocycle is always less than or equal to half of the top exponent of the Rauzy-Veech-Zorich cocycle. 

\begin{conj}\cite[\S 10]{kontsevich-Zorich} For the hyperelliptic components $\mathcal{M}^{\mathcal{H}}_{2g-2}$ and $\mathcal{M}^{\mathcal{H}}_{(g-1,g-1)}$ 
\begin{equation}
    \lim \limits_{g \to \infty} \theta_2 = 1.
\end{equation}
For all other components and other moduli spaces of holomorphic differentials
\begin{equation}
    \lim \limits_{g \to \infty} \theta_2 = \frac{1}{2}
\end{equation} 
\end{conj}
We need the following proposition that is due to Bufetov and Solomyak.

\begin{prop}(\cite[Corollary 3.8]{Bufetov-Solomyak-spectralcocycle2020}) 
Let $\chi_1$ be the top exponent of the Zorich cocycle (with respect to $\mu_Z$) and $\tilde{\chi}^{+}=\int \tilde{\chi}^+(\lambda,\pi,\zeta)d(\mu_Z\times m_{\mathbb{T}^d})$ be the average of the top Lyapunov exponent of the twisted cocycle with respect to $\mu_Z \times m_{\mathbb{T}^d}$. Then, 
\begin{equation} 
     \tilde{\chi}^{+} \leq \frac{\chi_1}{2}.
\end{equation}
\end{prop}

\begin{thm}\label{Kontsevitch-Zorich-Conjecure} There exists a family of $SL(2, \mathbb{R})$-invariant submanifolds $\mathcal{M}_p$ (for every prime $p$) of surfaces of genus going to infinity such that
\begin{equation}
    \limsup \limits_{p \to \infty} \theta_2(\mathcal{M}_p) \leq 1/2.
\end{equation}
\end{thm}
\begin{proof} Let $d=4$ and $\pi = \begin{pmatrix}
    1 2 3 4\\
    4 3 2 1
\end{pmatrix}.$ 
For a prime $p$ let $\hat{\lambda} \in \Delta^{\mathcal{A}_t}$ be such that for Lebesgue almost every $s$, $(\lambda_s, \pi):= \mathcal{F}(\hat{\lambda})$ has the following properties
\begin{itemize}
    \item $(\lambda_s,\pi)$ is Oseledets regular (see Subsection \ref{basic-ergodic-theory}) for Rauzy-Veech-Zorich cocycle with respect to $\mu_Z$, 
    \item for all $\zeta \in Q_p$, $(\lambda_s, \pi, \zeta)$ is Oseledets regular for the twisted cocycle with respect to $\mu_Z \times \nu_p$.
\end{itemize}
Let $\mathcal{M}_p:= \overline{SL(2, \mathbb{R}). X(\n, p, 
\hat{\lambda})}$ for some fixed $\n \in Q_p$. Let $s \in [0,1]$ be such that the above two conditions hold and $\spn(\lambda_s, \pi)$ is the return map of the translation flow 
on $X(\n, p, \hat{\lambda})$ in an Oseledets regular direction. Let $f \in \mathbb{C}^{(p-1)d+1}$ be a locally constant function with 
respect to $\spn( \lambda_s, \pi)$ of mean zero. By Lemma \ref{Small kernel surjection}, we can express $f$ as
\begin{equation}
    f= \sum_{j=0}^{p-1} \Bir_j(f_j),
\end{equation}
where $f_j$'s are locally constant functions with respect to $(\lambda_s, \pi)$ and $f_0$ has zero average. Then for every $x\in I$
\begin{equation}
     \limsup\limits_{N\to \infty} \frac{\log \big\vert S_N(f,x)\big\vert}{\log N}  \leq \max_{0\leq j\leq p-1}
     \bigg\{\limsup_{N \to \infty} \frac{\log\big\vert S_N(f_0,x)\big\vert}{\log N}, \ldots, \limsup_{N \to \infty} \frac{\log \big\vert S_N(f_{p-1}, \frac{p-1}{p} \n, x)\big\vert}{\log N}\bigg\},
\end{equation}
which implies that
\begin{equation}
    \theta_2(\mathcal{M}_p)  \leq \max \bigg\{ \frac{\chi_2}{\chi_1}, \max_{1\leq j \leq p-1} \frac{\tilde{\chi}^{+}\big(\lambda_s, \pi, \frac{j}{p}\n\big)}{\chi_1}\bigg\}.
\end{equation}
Thus, we need only show that the limsup of the second quantity in the curly brackets is less than or equal to $1/2$ as by \cite[Corollary 2]{Eskin-Kontsevich-Zorich} we know that $\chi_2/\chi_1=1/3$. 

By Oseledets regularity $\tilde{\chi}^{+}(\lambda_s, \pi, \frac{j}{p}\n)$ is the top exponent of the twisted cocycle with respect to one of the 
ergodic components of $\mu_Z \times \nu_p$ of which there are only $1/m_0$ many by Proposition  \ref{large-ergodic-components} (in this case $d=2g=4$). 

For the sake of contradiction, assume that 
\begin{equation}
    \limsup\limits_{p\to \infty} \theta_2(\mathcal{M}_{p}) >\frac12.
\end{equation}
Then, we may take a subsequence $\{p_i\}_{i=1}^\infty$ for which the limit exists and is bigger than 1/2 with the property that the number of ergodic components of the measures $\mu_Z \times \nu_{p_i}$ is constant equal to $k< \frac{1}{m_0}$. For the sake of simplicity of the notation, we denote this subsequence by $\{p_i\}$ as well. Denote by $\mu_{i,l}$, $1 \leq l \leq k$ the ergodic components 
of $\mu_Z \times \nu_{p_i}$. Then by Proposition  \ref{large-ergodic-components}, 
we have
\begin{equation}
    \mu_Z \times \nu_{p_{i}}= \sum_{l=1}^{k} c_{il} \mu_{il}
\end{equation}
where $1 \geq c_{il} \geq 1/m_0$ for every $l \in \{ 1, \ldots, k\}$. 
Therefore, by taking further subsequences, we can ensure that there are positive numbers $c_1,\ldots,c_k$ summing to $1$, and probability measures $\mu_1,\ldots,\mu_k$ such that $c_{il}\to c_l$ and $\mu_{il}\to \mu_l$  for every $1 \leq l \leq k$ as $i$ tends to infinity. Thus, we have
\begin{equation}
     \mu_Z \times m_{\mathbb{T}^4}= \lim_{i \to \infty} \mu_Z \times \nu_{p_{i}}= \sum_{l=1}^{k}c_l \mu_l.
\end{equation}
Since the limit of the measures $\mu_Z \times \nu_{p_{i}}$ is the ergodic measure $\mu_Z \times m_{T^{4}}$ and each $\mu_l$ is invariant under $\mathcal{R}^{Z}_{\mathbb{T}^4}$, we get  
\begin{equation}
    \mu_l= \mu_Z \times m_{\mathbb{T}^4}
\end{equation}
for every $l \in \{1, \ldots, k\}$. Thus, the upper semicontinuity of the top Lyapunov exponent yields that 
\begin{equation}
    \limsup\limits_{p\to \infty} \max_{1 \leq j \leq p-1}\tilde{\chi}^{+}\big(\lambda_s, \pi, \frac{j}{p}\n\big)= \limsup_{i \to \infty} \max_{1\leq l \leq k} \tilde{\chi}^{+}_{\mu_{il}} \leq \tilde{\chi}^{+}_{\mu_Z \times m_{\mathbb{T}^4}} \leq \frac{1}{2} \chi_1,
\end{equation}
which is a contradiction.
\end{proof}

\subsection{Discrepancy modulo primes}
\begin{definition} For an interval exchange transformation $T:I \to I$ and a subinterval $J \subset I$, we define the $n$-th step discrepancy of $J$ under $T$ by
\begin{equation}
    D_{J}(T, n):= \sup_{x \in I} |S_n(\one_J, x)-\int_{I} S_n(\one_J, y) dy|,
\end{equation}
where we recall that $S_n(f, x)$ stands for the Birkhoff sum under $T$ of the function $f$ evaluated at $x$. 
\end{definition}
\begin{thm}
    Let $\pi$ be an irreducible permutation of genus greater than $1$. Then, there exists a constant $\chi_{disc}>0$ such that for a $\mu_Z$-full measure set $\Delta_{disc} \subset \Delta$ and for every large enough prime $p$ and every $\lambda \in \Delta_{disc}$, we have
    \begin{equation}
        \limsup \limits_{n\to \infty} \sup\limits_{0\leq a<b\leq 1}\frac{\log |D_{[a,b]}(T_{\lambda, \pi}^{p}, n)|}{\log n}\geq \chi_{disc}.
    \end{equation}
\end{thm}
\begin{proof} 
Note that by (the proof of) Proposition \ref{periodic exponent approximation} and \ref{eq:lower bound 1-12} we get that there exists a positive exponent $\chi_{disc}>0$ such that for every large enough prime $p$ $\tilde{\chi}^{+}_{\mu_Z \times \nu_p}> \chi_{disc}$ (to be precise, this inequality holds for the average of the exponent along any fixed line corresponding to $\n$). Let $\zeta_p:= (1/p, \ldots, 1/p)$. Then there exists a subset $\Delta_{disc} \subset \Delta$ of full $\mu_Z$-measure such that for every $\lambda \in \Delta_{disc}$ and every $p\geq 5$, there exists a symbol $\alpha \in \mathcal{A}$ such that
\begin{equation}
    \limsup_{n \to \infty}\sup_{x\in I}\frac{|S_n(\one_\alpha, \zeta_p, x)|}{\log n} > \chi_{disc}
\end{equation}
Let $n$ be large enough so that 
\begin{equation}
    |S_n(\one_\alpha, \zeta_p, x)| \geq n^{\chi_{disc}}
\end{equation}
for some $x\in I$. Let $n=pk+r$ for some $k\in \mathbb{N}$ and $0\leq r \leq p-1$. Then, one can rewrite
\begin{align}
    |S_n(\one_\alpha, \zeta_p, x)| &\leq \sum_{j=0}^{p-1}\Big|\exp\big(\frac{2\pi is }{p}\big)\bigg(\sum_{s=0}^{k-1}\one_{\alpha} (T^{j+sp}(x))-k |I_{\alpha}|\bigg)\Big|+r\leq \\
    & \leq p D_{I_\alpha}(T^{p}, k) +r.
\end{align}
It implies the desired result.
\end{proof}

\subsection{Spectral dimension}
\begin{definition}(\it Spectral measure) Let $T_t: (X, \mu) \to (X, \mu)$ be a measure preserving flow and $f\in L^2(X, \mu)$ an observable. The spectral measure corresponding to $f$ is the unique measure $\sigma_f$ on $\mathbb{R}$ with the property that
\begin{equation}
    \hat{\sigma}_f(-\xi) =  \int f\circ T_t(x)f(x) d\mu(x),
\end{equation}
where $\hat{\sigma}_f(-\xi)= \int \exp(2\pi i t\xi) d\sigma_f(t)$ is the Fourier coefficient of $\sigma_f$ at phase $\xi$. 
\end{definition}
\begin{definition}(\it Local dimension of a measure) Let $\sigma$  be a finite measure on a metric space $X$. Then, the lower local dimension of $\sigma$ at $\omega$, denoted by $\underline{d}(\sigma, \omega)$, is defined as
\begin{equation}
    \underline{d}(\sigma, \omega):= \liminf_{r \to 0} \frac{\log \big\vert\sigma(B_r(\omega))\big\vert}{\log r},
\end{equation}
where $B_r(\omega)$ denotes the ball of radius $r$ centered on $\omega$. 
\end{definition}

Our main result, along with \cite[Theorem 3.6]{Bufetov-Solomyak-spectralcocycle2020}, implies the following.
\begin{cor} Let $\pi$ be an irreducible permutation of genus greater than $1$. Then, for almost every $(\lambda, \pi, h) \in \Delta \times \mathfrak{R} \times \mathbb{R}^{d}_{+}$ there exists a function $f$ defined on the suspension space obtained by $T_{\lambda, \pi}$ and the roof function $h$ such that for almost every $s\in \mathbb{R}$

\begin{equation}
    \underline{d}(\sigma_f, s)= 2- \frac{2\tilde{\chi}^{+}}{\chi_1} \geq 1.
\end{equation}
The same result holds true for almost every $(\lambda, \pi, h) \in \Delta \times \mathfrak{R} \times H^{+}(\pi)$. 
\end{cor}

\addcontentsline{toc}{section}{References}
     \bibliographystyle{alphaurl}

\end{document}